\documentclass[10pt]{article}
\usepackage{amsmath,amssymb}
\usepackage{enumerate}

\topskip=\baselineskip
\allowdisplaybreaks[1]

\newtheorem{theorem}{Theorem}[section]

\newtheorem{lemma}[theorem]{Lemma}
\newtheorem{proposition}[theorem]{Proposition}
\newtheorem{corollary}[theorem]{Corollary}
\newtheorem{definition}[theorem]{Definition}
\newtheorem{remark}[theorem]{Remark}
\newtheorem{example}[theorem]{Example}

\newcommand{\Z}{\mathbb{Z}}

\newcommand{\id}{\operatorname{id}}

\newcommand{\Sym}{\operatorname{Sym}}

\newcommand{\GK}{\operatorname{GKdim}}
\newcommand{\cl}{\operatorname{clKdim}}
\newcommand{\Aut}{\operatorname{Aut}}

\newcommand{\Ret}{\operatorname{Ret}}
\newcommand{\gr}{\operatorname{gr}}
\newcommand{\SD}{\operatorname{SD}}
\newcommand{\Mod}{\operatorname{mod}}

\newcommand{\free}[1]{\langle #1 \rangle}

\newenvironment{proof}{\par\noindent{\bf Proof.}}{$\qed$\par\bigskip}
\newcommand{\qed}{\enspace\vrule  height6pt  width4pt  depth2pt}
\usepackage{color}

\begin{document}

\title{Set-theoretic solutions of the Yang--Baxter equation, associated quadratic algebras and the minimality condition\thanks{The first author was partially supported by the grants
MINECO-FEDER  MTM2017-83487-P and AGAUR 2017SGR1725 (Spain). The
second author is supported in part by Onderzoeksraad of Vrije
Universiteit Brussel and Fonds voor Wetenschappelijk Onderzoek
(Belgium). The third author is supported by the National Science
Centre  grant 2016/23/B/ST1/01045 (Poland).}}
\author{F. Ced\'o \and E. Jespers \and J. Okni\'{n}ski}
\date{}

\maketitle

\begin{abstract}
Given a finite non-degenerate set-theoretic solution $(X,r)$ of
the Yang-Baxter equation and a field $K$, the structure
$K$-algebra of $(X,r)$  is $A=A(K,X,r)=K\langle X\mid xy=uv \mbox{
whenever }r(x,y)=(u,v)\rangle$. Note that $A=\oplus_{n\geq 0} A_n$
is a graded algebra, where $A_n$ is the linear span of all the
elements $x_1\cdots x_n$, for  $x_1,\dots ,x_n\in X$. One of the
known results asserts that the maximal possible value of $\dim
(A_2)$ corresponds to involutive solutions and implies several
deep and important properties of $A(K,X,r)$. Following recent
ideas of Gateva-Ivanova \cite{GI2018}, we focus on the minimal
possible values of the dimension of $A_2$. We determine lower
bounds and completely classify solutions $(X,r)$ for which these
bounds are attained in the general case and also in the
square-free case. This is done in terms of the so called derived
solution, introduced by Soloviev and closely related with racks
and quandles. Several problems posed in \cite{GI2018} are solved.
\end{abstract}

\section{Introduction}

Drinfeld \cite{drinfeld} in 1992  initiated the investigations of
the set-theoretic solutions of the Yang-Baxter equation on a
non-empty set $X$. Let $(X,r)$ be a quadratic set, that is, $r$ is a
bijective map $r:X\times X \rightarrow X\times X$. For $x,y\in X$ we
put $r(x,y) =(\sigma_x (y), \gamma_y (x))$. One says that $(X,r)$ is
non-degenerate if all the maps $\sigma_x$ and $\gamma_y$ are
bijective and $(X,r)$ is said to be involutive if
 $r^{2}=\id$.
If
  $$r_1 r_2 r_1 =r_2 r_1 r_2,$$
where $r_1=r\times \id_X$ and $r_2=\id_X\times r$ are maps from
$X^3$ to itself, then $(X,r)$ is said to be a braided set.
A braided set also is said to be  a set-theoretic solution of the Yang-Baxter
equation (see \cite{ESS, GIVdB}).

Many fundamental results in this area have been obtained by
Etingof, Schedler and Soloviev, Gateva-Ivanova and Van den Bergh,
Lu, Yan and Zhu; we refer the reader to \cite{ESS, GIVdB, LYZ,
Soloviev}. In \cite{ESS} Etingof, Schedler and Soloviev and in
\cite{GIVdB} Gateva-Ivanova and Van den Bergh  introduced the
following algebraic structures associated to a set-theoretic
solution $(X,r)$: the structure group $G(X,r)=\gr(X\mid
xy=uv\text{ if }r(x,y)=(u,v))$, the structure monoid
$M(X,r)=\free{X \mid xy=uv\text{ if }r(x,y)=(u,v)}$ and, for a
field $K$,  the structure algebra $A(K,X,r)$. The latter is the
$K$-algebra generated by $X$ with the ``same'' defining relations
as the monoid $M(X,r)$. Clearly $A(K,X,r)=K[M(X,r)]$, a monoid
algebra.  Of course, to define all these algebraic structures it
is sufficient that $(X,r)$ is a quadratic set.

The investigations have intensified since the discovery of  new
algebraic structures that are related to a set-theoretic solution
and that determine all non-degenerate  bijective (involutive)
solutions on a finite set (see for example \cite{B4, BCJ, Ven1,
Rump}).  Namely, to deal with involutive non-degenerate solutions
Rump \cite{Rump} introduced the algebraic structure called a
(left) brace and to deal with arbitrary non-degenerate solutions
Guarnieri and Vendramin \cite{Ven1} extended this notion to that
of a skew brace. In recent years skew braces have lead to new
methods that gave progress in the area of the Yang-Baxter
equation, in particular allowing to construct several new families
of solutions \cite{B3,BCJO,BCJO2,CJO2,SmSm}, but also allowed to
answer some questions in group and ring theory; see for example
Bachiller \cite{B2}, Bachiller, Ced\'o  and Vendramin \cite{BCV},
Catino, Colazzo and Stefanelli \cite{CCS},   Childs \cite{Ch},
Lebed \cite{Lebed}, Smoktunowicz \cite{Smok},  Smoktunowicz and
Vendramin \cite{SV}.

Suppose $(X,r)$ is a finite quadratic set. Since the defining
relations are homogeneous, the structure algebra $A=A(K,X,r)$ is
quadratic and  has a natural gradation (by length of elements in
$M(X,r)$) and clearly $A$ is a connected graded $K$-algebra. So,
$A=\oplus_{n\geq 0}A_n$ is generated by $A_1 =\mbox{span}_{K}(X)$
and each $A_n$ is finite dimensional. For more information on
quadratic algebras we refer to \cite{polpos}.

In the last two decades there was an intense study of monoids,
groups  and algebras associated with involutive quadratic sets, see
for example \cite{CJO04, CJO06, GI96, GI2012, GIAdv2018, GIJO,
JObook}.  An important particular case of these are the structure
monoid and algebra of an involutive non-degenerate set-theoretic
solution $(X,r)$ of the Yang-Baxter equation \cite{ESS, GIVdB}. In
this case, if $X$ is finite of cardinality $n$, then there are $n$
singleton orbits and $n\choose 2$ orbits of cardinality $2$ in $X^2$
under the action of the group $\langle r\rangle$ (we simply call
them $r$-orbits) and $\dim A(K,X,r)_{2} =n +{n\choose 2}$. Moreover,
the structure algebra $A(K,X,r)$, over a field $K$, shares many of
the properties of the polynomial algebra $K[x_1,\dots,x_n]$. In
particular, it is a domain, it satisfies a polynomial identity and
it has Gelfand-Kirillov dimension $n$. The structure monoid $M(X,r)$
is embedded in the structure group $G(X,r)=\gr((x,\sigma_x) \mid
x\in X) \subseteq \Z^{n}\rtimes \Sym_{X}$, where $\Sym_{X}$ acts
naturally on the free abelian group $\Z^n$ of rank $n$. Also, in
this case, the structure group $G(X,r)$ is solvable, torsion-free
and a Bieberbach group with a free abelian subgroup of rank $n$ of
finite index.

Gateva-Ivanova and Majid \cite{GIM}, Guarnieri and Vendramin in
\cite{Ven1}, Lebed and Vendramin in \cite{LebedVendramin} and
Bachiller in \cite{B4} continued the study of the structure of
$M(X,r)$ and $G(X,r)$ for arbitrary finite non-degenerate braided
sets $(X,r)$ initiated by Lu, Yau and Zhu in \cite{LYZ} and
Soloviev in \cite{Soloviev}. Such solutions were linked to self
distributive solutions (abbreviated $\SD$)  in
\cite{ESS,LebedVendramin,LYZ,Soloviev}. Recall from \cite{GI2018}
that a quadratic set $(X,r)$ is said to be $\SD$ if all $\gamma_y
=\id$, that is $r(x,y)=(\sigma_x(y),x)$, for $x,y \in X$.

Recall that a rack is a set $X$ with a binary operation $\lhd$ such
that $x\lhd (y\lhd z)=(x\lhd y)\lhd (x\lhd z)$, for all $x,y,z\in
X$, and the maps $y\mapsto x\lhd y$ are bijective for all $x\in X$.
Racks are related with pointed Hopf algebras \cite{AG}. A quandle is
a rack $(X,\lhd)$, such that $x\lhd x=x$ for all $x\in X$. Quandles
were introduced by Joyce in \cite{J1} and provide an important tool
in knot theory (see for example, \cite{CS,LV}).  It is known and
easy to prove that $(X,r)$  is a non-degenerate $\SD$ braided set if
and only if $(X,\lhd)$ is a rack, where $x\lhd y=\sigma_x(y)$ for
all $x,y\in X$. Furthermore, quandles correspond to non-degenerate
square-free $\SD$ braided sets.

Given a non-degenerate braided set $(X,r)$, Soloviev in
\cite{Soloviev} introduced its derived solution $(X,r')$  as
follows.  Write $r(x,y)=(\sigma_x (y), \gamma_y (x))$ for all
$x,y\in X$. Then
 $$r'(x,y) =(\gamma_x (\sigma_{\gamma^{-1}_{y}(x)}(y)), x)$$
for all $x,y\in X$. Put $\sigma'_x(y) = \gamma_x
(\sigma_{\gamma^{-1}_{y}(x)}(y))$ for $x,y\in X$. Then, for all
$x,y\in X$,
  $$r'(x,y)=(\sigma'_x (y),x)$$
and $(X,r')$ is  a  non-degenerate $\SD$ braided set (see
\cite[Theorem~2.3]{Soloviev}). The structure monoids $M(X,r)$ and
$M(X,r')$, as well as their respective algebras, are strongly
related  (see \cite[Proposition~1.4]{JKV}). Indeed, there exists
an action $\theta : M(X,r) \rightarrow \Aut(M(X,r'))$ and a
bijective $1$-cocycle $\phi: M(X,r) \longrightarrow M(X,r')$ with
respect to $\theta$, that is, $\phi (ab) =\phi (a) \theta (a)(\phi
(b))$ for $a,b\in M(X,r)$, satisfying $\theta (x)(y) =\sigma_x(y)$
and $\phi (x) =x$ (for all $x,y\in X$). Furthermore, $\phi$ is
degree preserving and $f(x) =(\phi (x), \theta (x))$ for $x\in
M(X,r)$ defines an injective monoid morphism
$$f:M(X,r) \longrightarrow M(X,r') \rtimes \Aut(M(X,r')).$$ In
particular,
\begin{eqnarray*}
M(X,r) &\cong& f(M(X,r))\;
     =\; \{ (\phi(a),\theta (a)) \mid a\in M(X,r)\} \subseteq
  M(X,r') \rtimes \Aut(M(X,r')).
  \end{eqnarray*}

Suppose that $(X,r)$ is a finite non-degenerate braided set.
Put $M=M(X,r)$ and $M'=M(X,r')$. Both monoids are graded monoids
(graded by the length of elements). So, $M=\bigcup_{m}M_m$, where
$M_m$ denotes  the set of elements of $M$ of length $m$.
Similarly, $M'=\bigcup_{m}M'_m$. Clearly, the structure algebras
$A=A(K,X,r)=\oplus_{m\geq 0} A_m$ and $A'=A(K,X,r')=\oplus_{m\geq
0}A'_{m}$ are connected graded $K$-algebras with
    $$\dim (A_m) =\dim (A'_{m})=|M_m| =|M'_m|,$$
for all non-negative  integers $m$.  In  particular, $\dim (A_2)
=\dim (A'_2)$ is the number of $r$-orbits in $X^2$ and the number of
$r'$-orbits in $X^2$. In the monoid $M'=M(X,r')$ we have that
$M'a=aM'$ for all $a\in M'$. Hence $|M'_2|$ is at most the number of
words of length $2$ in the free abelian monoid of rank $|X|=n$. Thus
  $$\dim A_2  \leq n +{n\choose 2}.$$
It is easy to describe when this upper bound is reached; we call this the
maximality condition for $A(K,X,r)$ (and $M(X,r)$). For this,
denote by $O_{(x,y)}$ the $r'$-orbit of $(x,y)\in X^2$ and assume
$\dim A_2 = n+{n\choose 2}$. Then
 $n^2 =\sum_{i=1}^{n +{n\choose 2}} | O_{(x_i ,y_i)}|$, for some $(x_i,y_i)\in X^2$.
One clearly has that $r'(x,y)=(\sigma'_x (y), x)\neq (x,y)$ if $x$
and $y$  are distinct elements in $X$. Hence, for such elements
$|O(x,y)|\geq 2$. Thus, if there are $m$ orbits with one element,
then $n\geq m$ and
 $$n^2 \geq  m + 2(n +{n\choose 2} -m) = n(n+1) -m=n^2+ n-m.$$
Hence, $n=m$ and thus $|O_{(x,x)}|=1$ for all $x\in X$, and all
other orbits have precisely $2$ elements. Therefore $r'$ is
involutive and each $\sigma_x =\id$, that is, $M(X,r')$ is the
free abelian monoid of rank $n$. So, the maximality condition
holds precisely when $r'$  is involutive and $M(X,r')$ is the free
abelian monoid of rank $|X|$. The following result  is a
consequence of Theorem~4.5 in \cite{JKV}.

 \begin{theorem}
 Let $(X,r)$ be a finite non-degenerate braided set and $K$ a field.
 The following properties are equivalent:
(1)  $A(K,X,r)$ satisfies the maximality condition, that is $\dim
(A(K,X,r)_2) = {|X| \choose 2} +|X|$;
 (2) $(X,r)$ is involutive;
 (3) $M(X,r) $ is a cancellative monoid;
 (4) $M(X,r')$ is free abelian;
 (5) $A(K,X,r)$ is a prime algebra;
 (6) $A(K,X,r)$ is a domain;
 (7) $\cl (A(K,X,r)) =|X|$;
 (8) $\GK ( A(K,X,r)) =|X|$.
 \end{theorem}

This result confirms that for a finite non-degenerate braided
set $(X,r)$ the maximality condition is determined by strong
properties on the set-theoretic solution $(X,r)$ and, equivalently, on
the algebra $A(K,X,r)$.

Note that the structure of $A(K,X,r)$ has been extensively
studied for finite quadratic sets $(X,r)$ provided $r$ is
involutive and for each $x\in X$ there exists a unique $y$ such that
$r(x,y)=(x,y)$  (see for example
\cite{CJO04,CJO06,CO,GI2012,GIJO,JObook,JVC}). In this case,
$\dim(A(K,X,r)_2)={n \choose 2} +n$, where $|X|=n$. If $(X,r)$ is
also non-degenerate, then the monoid $M(X,r)$ is called quadratic
monoid in \cite{JVC}. If, furthermore,  $(X,r)$ is square-free then
the monoid
$M(X,r)$ is called of skew type in \cite{JObook}.

In \cite{GI2018}, Gateva-Ivanova introduced a `minimality
condition' for a finite non-degenerate braided set $(X,r)$, with a
focus on square-free sets. In other words, she proposed to
consider the case where $\dim A_2$ is smallest possible. Several
interesting questions are posed concerning the lower bound on
$\dim A_2$ and its impact on the combinatorial properties of
$(X,r)$. By analogy with the maximality condition, one is tempted
to expect that this can also lead to nontrivial structural
properties of the solution $(X,r)$ and of its associated algebraic
structures (monoids, groups and algebras).  This viewpoint turned
out to be fruitful also in other contexts of the theory of
quadratic algebras, see Sections 6.4 and 6.5 in \cite{polpos}.

In this paper we answer several of the questions posed in
\cite{GI2018}. Our approach relies on a reduction to $\SD$ braided
sets, based on the recent result in \cite{JKV}, described above. The
main focus is on finding the exact lower bounds on $\dim(A_2)$ and
then on characterizing all sets $(X,r)$ for which these bounds are
realized.

We list the main results of the paper. First, we show in
Example~\ref{dim-An-one} that there exists a finite non-degenerate
quadratic set $(X,r)$ such that $\dim(A_m)=1$, for all $m>1$; so,
in particular, $\GK(A)=1$.

If, furthermore, $(X,r)$ is a braided set then we prove in
Theorem~\ref{GeneralMinDim} that $\dim (A_2)\geq \frac{|X|}{2}$.
Moreover,  we describe explicitly all cases when the lower bound
$\frac{|X|}{2}$ or $\frac{|X|+1}{2}$ (depending on whether $|X|$
is even or odd) is reached, in terms of the derived solution
$(X,r')$. In particular, this answers  Question 3.18(1) in
\cite{GI2018}. As a consequence, it follows in these cases that
the solution $(X,r)$ is indecomposable.

One of the main cases that arises in this context is when $(X,r')$
is the braided set associated to the so called dihedral
quandle. We discuss this case in detail in Section~\ref{quandles},
also in the context of the impact of the size of $A_2$ on the growth
function of the algebra $A$.

For square-free non-degenerate quadratic sets $(X,r)$ it is easy
to see that $|X|+1$ is the smallest possible value of $\dim(A_2)$.
In Example~\ref{dim-An-sf} we present a construction  with
$\dim(A_m)=|X|+1$ for all $m>1$; so that $\GK(A)=1$ also in this
case.

Let $(X,r)$ be a finite square-free non-degenerate braided set. In
Theorem~\ref{SquareFreeMinDim} we prove that $\dim (A_2) \geq 2
|X|-1$. Then the main result of Section~\ref{sf-braided},
Corollary~\ref{cor-minimal}, shows that this lower bound is
achieved if and only if the derived set $(X,r')$ is of one of the
following types:
\begin{enumerate}
\item $|X|$ is an odd prime and $(X,r)$ is the braided set associated to the dihedral
quandle,
\item $|X|=2$ and $(X,r)$ is the trivial braided set,  that is
$r(x,y)=(y,x)$ for all $x,y\in X$,
\item $X=\{ 1,2,3\}$ and $\sigma_1=\sigma_{2}=\id$, $\sigma_3=(1,2)$.
\end{enumerate}

In particular, this answers another problem of Gateva-Ivanova
(Question~3.18(2) in \cite{GI2018}). Note that the dihedral quandles
of cardinality an odd prime are simple quandles (see \cite{J2}).
Thus this also answers Open Questions~4.3.1 (1) and (2) in
\cite{GI2018}. On the other hand, the non-singleton orbits in $X^3$
under the action of the group $\langle r_1,r_2\rangle$ in the above
type (3) have not the same cardinality. In fact, $|O_{(1,1,3)}|=12$,
$|O_{(1,3,3)}|=6$ and $|O_{(1,1,2)}|=|O_{(1,2,2)}|=3$. Therefore
this answers in the negative Open Question 4.1.2 in \cite{GI2018}.

We conclude with a solution of Problem 4.3.2 of \cite{GI2018},
concerning the study of all finite $2$-cancellative, square-free
$\SD$ braided sets $(X,r)$, with $r(x,y)=(\sigma_x(y),x)$, such that
$\sigma_x^2=\id$, for all $x\in X$, that have a  given $r$-orbit in
$X^2$ of cardinality $|X|$; this is done in Proposition~\ref{sd}.

\section{Preliminaries}

In this section we recall some basic notions and we prove certain
auxiliary results that are useful when dealing with $r$-orbits in
$X^2$.
\begin{definition} A quadratic set is a pair $(X,r)$ of a nonempty set
$X$ and a bijective map $r\colon X\times X\longrightarrow X\times
X$. We will write $r(x,y)=(\sigma_x(y),\gamma_y(x))$ for all $x,y\in
X$.
\end{definition}

\begin{definition}
Let $(X,r)$ be a quadratic set. We say that $(X,r)$ is
non-degenerate if the maps $\sigma_x$ and $\gamma_x$ are bijective
for all $x\in X$. We say that $(X,r)$ is involutive if
$r^2=\id_{X^2}$. We say that $(X,r)$ is square-free if
$r(x,x)=(x,x)$, for all $x\in X$. We say that $(X,r)$ is a braided
set if
$$r_1r_2r_1=r_2r_1r_2,$$
where $r_1=r\times \id_X$ and $r_2=\id_X\times r$ are maps from
$X^3$ to itself.
\end{definition}

A braided set is a set-theoretic solution of the Yang-Baxter
equation (see \cite{ESS, GIVdB}).

\begin{definition}
Let $(X,r)$ be a quadratic set. The monoid associated with $(X,r)$
is the monoid $M(X,r)=\langle X\mid
xy=\sigma_x(y)\gamma_y(x)\rangle$. We say that $(X,r)$ is
$2$-cancellative if in $M(X,r)$
$$xy=xz \text{ or } yx=zx\text{ implies }y=z,$$
for all $x,y,z\in X$.
 Let $K$ be a field. The
$K$-algebra associated with $(X,r)$ is the monoid algebra
$A(K,X,r)=K[M(X,r)]$. For every nonnegative integer $i$, we define
$A_i=Span_K\{x_1\cdots x_i\mid x_1,\dots ,x_i\in X\}$ in
$A(K,X,r)$.
\end{definition}

The monoid $M(X,r)$ and the algebra $A(K,X,r)$ are also called the
structure monoid and the structure $K$-algebra of $(X,r)$.

\begin{remark}\label{monoid}
{\rm Note that if $(X,r)$ is a quadratic set then $M(X,r)$ is
naturally graded by the length. Since $r\in \Sym_{X^2}$, it is clear
that $xy=zt$, in $M(X,r)$, for $x,y,z,t\in X$, if and only if
$(x,y)$ and $(z,t)$ belong to the same $r$-orbit in $X^2$. Let
$r_{m,i}\colon X^m\longrightarrow X^m$ be the map
$r_{m,i}=\id^{i-1}\times r\times \id^{m-i-1}$, for $1\leq i<m$. Then
$x_1\cdots x_m=y_1\cdots y_m$ in $M(X,r)$, for $x_i,y_i\in X$, if
and only if $(x_1,\dots ,x_m)$ and $(y_1,\dots, y_m)$ belong to the
same orbit of $X^m$ under the action of the subgroup of $\Sym_{X^m}$
generated by $r_{m,1},\dots ,r_{m,m-1}$. We will denote the orbit of
$(x_1,\dots ,x_m)$ by $O_{(x_1,\dots ,x_m)}$.}
\end{remark}

The following result is well-known and its proof is
straightforward.

\begin{lemma}\label{conditions}
Let $(X,r)$ be a quadratic set. We write
$r(x,y)=(\sigma_x(y),\gamma_y(x))$ for all $x,y\in X$. Then  $(X,r)$
is a braided set if and only if the following conditions hold.
\begin{itemize}
\item[(i)]
$\sigma_{x}\sigma_{y}=\sigma_{\sigma_x(y)}\sigma_{\gamma_y(x)}$,
\item[(ii)]
$\gamma_{x}\gamma_{y}=\gamma_{\gamma_x(y)}\gamma_{\sigma_y(x)}$,
\item[(iii)]
$\gamma_{\sigma_{\gamma_{x}(y)}(z)}\sigma_{y}(x)=\sigma_{\gamma_{\sigma_{x}(z)}(y)}\gamma_{z}(x)$,
\end{itemize}
for all $x,y,z\in X$.
\end{lemma}

The next technical result is very useful to study the $r$-orbits in
$X^2$ for non-degenerate braided sets $(X,r)$.

\begin{lemma} \label{inductive}
Let $(X,r) $ be a non-degenerate braided set such that $r(x,y) =
(\sigma_{x}(y), x)$ for all $x,y \in X$. Then
\begin{enumerate}
\item
$\sigma_{(\sigma_{x}\sigma_{y})^{k}\sigma_{x}^{-k}(x)}=(\sigma_{x}\sigma_{y})^{k}
\sigma_{x}(\sigma_{x}\sigma_{y})^{-k}$,
\item $\sigma_{(\sigma_{x}\sigma_{y})^{k}\sigma_x\sigma_{y}^{-k}(y)}=(\sigma_{x}\sigma_{y})^{k}
\sigma_{x}\sigma_{y}\sigma_{x}^{-1}(\sigma_{x}\sigma_{y})^{-k}$,
\item $r^{2k}(x,y) = ((\sigma_{x}\sigma_{y})^{k}\sigma_{x}^{-k}(x),
(\sigma_{x}\sigma_{y})^{k-1}\sigma_{x}\sigma_{y}^{-k+1}(y) )$,
\item $r^{2k+1}(x,y) = ( (\sigma_{x}\sigma_{y})^{k}\sigma_x \sigma_{y}^{-k}(y)),
(\sigma_{x}\sigma_{y})^{k}\sigma_{x}^{-k}(x))$
\end{enumerate}
for all non-negative integers $k$.
\end{lemma}
\begin{proof}
We shall prove the result by induction on $k$. For $k=0$,  part (1)
obviously holds and  part (2) follows from
Lemma~\ref{conditions}(i). Also part (3) and part (4) are obvious as
$(x,y)=(x,(\sigma_x \sigma_y)^{-1} \sigma_x \sigma_y (y)) =(x,y)$
and $r(x,y)=(\sigma_x (y),x)$. Assume now that the result holds for
$k$. Then
\begin{eqnarray*}
r^{2(k+1)}(x,y) &=& r(r^{2k+1})(x,y)\\
 &=&(\sigma_{(\sigma_{x}\sigma_{y})^{k}\sigma_x \sigma_{y}^{-k}(y)} (\sigma_x \sigma_y)^{k} \sigma_x^{-k}(x), (\sigma_x \sigma_y )^{k} \sigma_x \sigma_y^{-k}(y))\\
&=& ((\sigma_x \sigma_y)^{k} \sigma_x \sigma_y \sigma_x^{-1} (\sigma_x \sigma_y)^{-k}  (\sigma_x \sigma_y )^{k} \sigma_x^{-k}(x), (\sigma_x \sigma_y )^{k} \sigma_x \sigma_y^{-k}(y))\\
 &=& ((\sigma_x \sigma_y )^{k+1} \sigma_x^{-(k+1)}(x), (\sigma_x \sigma_y )^{k} \sigma_x \sigma_y^{-k}(y))
\end{eqnarray*}
Hence part (3) follows for $k+1$. Furthermore, because of
Lemma~\ref{conditions}.(i),
 $$\sigma_{(\sigma_x \sigma_y)^{k} \sigma_x \sigma_y^{-k}(y)}
  \sigma_{(\sigma_x \sigma_y)^{k}\sigma_x^{-k}(x)} =
  \sigma_{(\sigma_x \sigma_y)^{k+1}\sigma_x^{-(k+1)}(x)}
  \sigma_{(\sigma_x \sigma_y )^{k} \sigma_x \sigma_y^{-k}(y)}.
  $$
Hence,
\begin{eqnarray*}
\sigma_{(\sigma_x \sigma_y)^{k+1}\sigma_x^{-(k+1)}(x)} &=&
\sigma_{(\sigma_x \sigma_y)^{k} \sigma_x \sigma_y^{-k}(y)}
  \sigma_{(\sigma_x \sigma_y)^{k}\sigma_x^{-k}(x)}
   \sigma^{-1}_{(\sigma_x \sigma_y )^{k} \sigma_x \sigma_y^{-k}(y)}\\
   &=& (\sigma_{x} \sigma_y )^{k} \sigma_x \sigma_y \sigma_x^{-1} (\sigma_x \sigma_y )^{-k} \\
   && (\sigma_x \sigma_y)^{k} \sigma_x (\sigma_x \sigma_y )^{-k} (\sigma_x \sigma_y)^{k} \sigma_x \sigma_y^{-1} \sigma_x^{-1} (\sigma_x \sigma_y)^{-k}\\
   &=& (\sigma_x \sigma_y)^{k+1} \sigma_x (\sigma_x \sigma_y )^{-(k+1)}.
\end{eqnarray*}
Thus (1) follows for $k+1$. Further,
\begin{eqnarray*}
r^{2(k+1)+1}(x,y) &=& r(r^{2(k+1)}(x,y))\\
&=& (\sigma_{(\sigma_x \sigma_y)^{k+1} \sigma_x^{-(k+1)}(x)}
(\sigma_x \sigma_y)^{k} \sigma_x \sigma_y^{-k}(y),
(\sigma_x \sigma_y)^{k+1} \sigma_x^{-(k+1)}(y))\\
&=& ((\sigma_x \sigma_y)^{k+1} \sigma_x (\sigma_x \sigma_y)^{-(k+1)} (\sigma_x \sigma_y)^{k} \sigma_x \sigma_{y}^{-k}(y),
 (\sigma_x \sigma_y)^{k+1} \sigma_x^{-(k+1)}(x))\\
&=& ((\sigma_x \sigma_y)^{k+1} \sigma_x \sigma_y^{-(k+1)}(y),
  (\sigma_x \sigma_y)^{k+1}\sigma_{x}^{-(k+1)}(x)),
\end{eqnarray*}
and thus part (4) is true for $k+1$. Furthermore, because of
Lemma~\ref{conditions}.(i),
$$\sigma_{(\sigma_x \sigma_y)^{k+1} \sigma_x^{-(k+1)}(x)}
   \sigma_{(\sigma_x \sigma_y )^{k} \sigma_x \sigma_y^{-k}(y)}
=
   \sigma_{(\sigma_x \sigma_y)^{k+1} \sigma_x \sigma_y^{-(k+1)}(y)} \sigma_{(\sigma_x \sigma_y)^{k+1} \sigma_x^{-(k+1)} (x)}.
$$
Hence,
\begin{eqnarray*}
\lefteqn{\sigma_{(\sigma_x \sigma_y)^{k+1}
\sigma_x \sigma_{y}^{-(k+1)}(y)}}\\
&=& (\sigma_x \sigma_y)^{k+1} \sigma_x (\sigma_x \sigma_y)^{-(k+1)}
(\sigma_x \sigma_y )^{k} \sigma_x \sigma_y \sigma_x^{-1} (\sigma_x
\sigma_y)^{-k}
(\sigma_x \sigma_y)^{k+1} \sigma_x^{-1} (\sigma_x \sigma_y)^{-(k+1)}\\
&=&
(\sigma_x \sigma_y)^{k+1} \sigma_x \sigma_y \sigma_x^{-1}
(\sigma_x \sigma_y)^{-(k+1)}.
\end{eqnarray*}
Hence part (2) follows for $k+1$. Therefore, the result follows by
induction.
\end{proof}

We will also need the following direct consequence of
Lemma~\ref{inductive}.

\begin{corollary}\label{rn}
Let $(X,r)$ be a square-free non-degenerate braided set such that
$r(x,y)=(\sigma_x(y),x)$ for all $x,y\in X$. Then
\begin{eqnarray*}
r^{2k+1}(x,y)&=&((\sigma_x\sigma_y)^{k}\sigma_x(y),(\sigma_x\sigma_y)^{k}(x)),\\
r^{2k}(x,y)&=&((\sigma_x\sigma_y)^{k}
(x),(\sigma_x\sigma_y)^{k-1}\sigma_x(y)),
\end{eqnarray*}
for all non-negative integers $k$.
\end{corollary}

\section{Non-degenerate braided sets} \label{sect-non-deg}
In this section we give a lower bound for the dimension of
$A(K,X,r)_2$ for finite non-degenerate braided sets $(X,r)$ and we
describe when this lower bound is achieved.

Let $(X,r)$ be a finite quadratic set. It is clear that if $r\in
\Sym_{X^2}$ is any cycle of length $|X|^2$, then the number  of
$r$-orbits in $X^2$ is $1$.

\begin{example} \label{dim-An-one}
Let $K$ be a field. Let $n>1$ be an integer. Then there exists a
non-degenerate quadratic set $(X,r)$ such that $|X|=n$  and
$\dim(A_m)=1$, for all $m>1$. In particular, its associated
$K$-algebra $A=A(K,X,r)$ has $GK$-dimension $1$.
\end{example}

\begin{proof}
Let $X=\{ 1,2,\dots ,n\}$. Let $r\colon X\times X\longrightarrow
X\times X$  be the map defined by $r(x,y)=(\sigma_x(y), x)$, where
\begin{align*}&\sigma_1=\sigma_n=(1,2,\dots ,n)
\mbox{ and } \sigma_{x}=(x,x+1,\dots, n)
\end{align*}
for all $1<x<n$. Note that
\begin{eqnarray*}
&&(1,1)\overset{r}{\mapsto}(2,1)\overset{r}{\mapsto}
(1,2)\overset{r}{\mapsto} \cdots\overset{r}{\mapsto}
(1,n-1)\overset{r}{\mapsto} (n,1)\overset{r}{\mapsto}
(2,n)\\
&&\overset{r}{\mapsto} (2,2)\overset{r}{\mapsto}
(3,2)\overset{r}{\mapsto} (2,3)\overset{r}{\mapsto} \cdots
\overset{r}{\mapsto}(2,n-1)\overset{r}{\mapsto}
(n,2)\overset{r}{\mapsto} (3,n)\\
&&\overset{r}{\mapsto} (3,3)
\overset{r}{\mapsto} \cdots\\
&&\overset{r}{\mapsto}(n-1,n-1)\overset{r}{\mapsto}(n,n-1)\overset{r}{\mapsto}
(n,n)\overset{r}{\mapsto} (1,n) \overset{r}{\mapsto} (1,1).
\end{eqnarray*}
Hence $O_{(1,1)}=X^2$ and thus $\dim (A_2)=1$. It is easy to see
by induction on $m\geq 2$ that $\dim (A_m)=1$. Hence $\dim_K
(\oplus_{i=0}^mA_i)=n+m$, for every positive integer $m>1$.
Therefore $\GK(A)=1$ and this finishes the proof.
\end{proof}

\begin{remark} {\rm
In \cite{CJO06} (see also  \cite[Chapter~10]{JObook}) it is shown
that, for every integer $n>1$ and for every integer $2\leq i\leq n$,
there are examples of non-degenerate monoids of skew type $M(X,r)$,
with $|X|=n$ and such that $\GK(A(K,X,r))=i$, for any field $K$.
Furthermore, for every positive integer $k$,  there are
non-degenerate monoids of skew type $M(X,r)$ such that
$\GK(A(K,X,r))=1$ and $|X|=4^k$, for any field $K$ (see
\cite{CJO04} or \cite[Chapter~10]{JObook}). It is an open problem
whether, for every integer $n>4$, there exist non-degenerate monoids
of skew type $M(X,r)$ such that $|X|=n$ and $\GK(A(K,X,r))=1$, for
any field $K$. }
\end{remark}

Following \cite{GI2018}, we say that a quadratic set $(X,r)$ is
$\SD$ if $r(x,y)=(\sigma_x(y),x)$, for some maps $\sigma_x\colon
X\longrightarrow X$.

Note that if $(X,r)$ is an $\SD$ quadratic set, then, since $r$ is
bijective, for every $x,z\in X$ there exists a unique $y\in X$ such
that $r(x,y)=(\sigma_x(y),x)=(z,x)$. Hence the maps $\sigma_x$ are
bijective, for all $x\in X$, and thus $(X,r)$ is non-degenerate.

In order to state our next result it is convenient to recall the
notion of multipermutation solution. Initially this has been studied
for non-degenerate involutive braided sets \cite{ESS} and later it
has been studied by Lebed and Vendramin \cite{LebedVendramin} for
arbitrary non-degenerate braided sets. Let $(X,r)$ be a
non-degenerate braided set. Then the relation $\sim$, defined  on
$X$ by $x\sim y$ if and only if $\sigma_x =\sigma_y$ and $\gamma_x
=\gamma_y$, is an equivalence relation and it induces a
non-degenerate braided set $(X/\sim, \overline{r})$, called the
retraction and denoted by $\mbox{Ret}(X, r)$, where
$\overline{r}([x],[y])=([\sigma_x(y)],[\gamma_y(x)])$. A
non-degenerate braided set is called multipermutation  of level $n$
if $n$ is the minimal non-negative integer such that the underlying
set of $\mbox{Ret}^{n}(X, r)$ has cardinality one.  A non-degenerate
braided set is called retractable if $|X/\sim |<|X|$, and
irretractable otherwise.  By the definition of the retraction, it is
clear that if $(X,r)$ is $\SD$ then so is $\mbox{Ret}(X, r)$.

\begin{theorem}\label{GeneralMinDim}
Let $K$ be a field. Let $(X,r)$ be a finite non-degenerate $\SD$
braided set and let $A=A(K,X,r)$ be its structure $K$-algebra. Then,
$\dim (A_2)\geq \frac{|X|}{2}$.

Furthermore, if $|X|$ is even then the lower bound  $\frac{|X|}{2}$
is reached precisely when all $\sigma_{x}$, with $x\in X$, are equal
to a cycle $\sigma$ of length $|X|$. If $|X|$ is odd then the lower
bound $\frac{|X|+1}{2}$ is reached when  all $\sigma_{x}$, with
$x\in X$, are equal to a cycle $\sigma$ of length $|X|$. In
particular, $r(a,b) =(\sigma (b),a)$ for all $a,b\in X$ and thus the
solution $(X,r)$ is an indecomposable multipermutation solution of
level $1$.
\end{theorem}
\begin{proof}
Write $r(a,b)=(\sigma_{a}(b),a)$, for $a,b\in X$. Let $\mathcal{S}=
\{ \sigma_{a}
  \mid a \in X\} $.
For $\sigma \in \mathcal{S}$ put $X_{\sigma}=\{ a\in X \mid
\sigma_{a}=\sigma\}$. Clearly $X=\bigcup_{\sigma\in \mathcal{S}}
X_{\sigma}$, a disjoint union. By Lemma~\ref{conditions}, for
$a,b\in X$, we have $\sigma_a \sigma_b =\sigma_{\sigma_{a}(b)}
\sigma_a $. Let $\sigma \in \mathcal{S}$ and assume $a,b\in
X_{\sigma}$. Then, $\sigma_{\sigma_{a}(b)} =\sigma_a \sigma_b
\sigma_a^{-1} =\sigma^{2}\sigma^{-1}=\sigma$. Thus
$\sigma_{\sigma_{a}(b)} \in X_{\sigma}$. Therefore, if $r_{\sigma}$
denotes the restriction of $r$ to $X_{\sigma}^{2}$ then also
$(X_{\sigma},r_{\sigma})$  is a $\SD$ braided set. Write $\sigma =
c_1 c_2 \cdots c_s$ as a product of pairwise disjoint nontrivial
cycles in $\Sym_{X}$. For $1\leq j \leq s$, put
 $$X_{\sigma , j}=\{ a\in X_{\sigma} \mid c_{j}(a)\neq a \}
\; \mbox{ and }\;
  X_{\sigma , 0} \{ a\in X_{\sigma} \mid \sigma (a)=a \}.$$
We have that $X_{\sigma , j} \neq \emptyset$ for all $1\leq j \leq
s$. Note that, for $1\leq j \leq s$, the restriction $r_j$ of $r$ to
$X_{\sigma , j}^{2}$ satisfies
 $$r_j (a,b) =(c_j(b),a)$$
for all $a,b\in X_{\sigma , j}$ and $r(c,c) =(\sigma (c),c) =(c,c)$
for all $c\in X_{\sigma, 0}$. Hence, for all $c\in X_{\sigma, 0}$ we
have that $\{ (c,c)\}=O_{(c,c)}$, the $r$-orbit of $(c,c)$, and
$(X_{\sigma , j},r_j)$ is a $\SD$ braided set for all $1\leq j \leq
s$.

We claim that, under the action of $\langle r_j\rangle$, the number
of orbits in $X^2_{\sigma , j}$ is $\frac{|X_{\sigma , j}|}{2}$ if
$|X_{\sigma , j}|$ is even and otherwise it is
$\frac{|X_{\sigma ,j}|+1}{2}$. In order to prove the claim we may
assume that $X_{\sigma , j}=\Z/(m)$ and $r_j (a,b)=(b+1,a)$ for all
$a,b\in \Z/(m)$. We have
$$(0,1) \stackrel{r_j}{\mapsto} (2,0) \stackrel{r_j}{\mapsto} (1,2)
\stackrel{r_j}{\mapsto} (3,1) \stackrel{r_j}{\mapsto} \cdots
\stackrel{r_j}{\mapsto}  (m-1,0) \stackrel{r_j}{\mapsto} (1,m-1)
\stackrel{r_j}{\mapsto} (0,1).$$ By Lemma~\ref{inductive}
  $$r_{j}^{2k}(a,b) = (c_{j}^{k}(a),c_j^{k}(b))=(a+k,b+k)$$
  and
  $$r_{j}^{2k+1}(a,b) = (c_{j}^{k+1}(b),c_{j}^{k}(a)) =(b+k+1,a+k)$$
for all positive integers $k$. Let $k_0$ be the smallest positive
integer such that $r_j^{2k_0} (a,b)=(a,b)$ or $r_j^{2k_0
+1}(a,b)=(a,b)$. If $r_j^{2k_0} (a,b)=(a,b)$ then $a+k_0 =a$ and
$b+k_0=b$. Hence, $k_0=m$ in this case. If $r_j^{2k_0
+1}(a,b)=(a,b)$ then $b+k_0+1=a$ and $a+k_0=b$. Thus
$a-b=k_0+1=b-a+1$ in $\Z/(m)$. Hence $2(a-b)=1$ and $k_0=b-a$ in
$\Z/(m)$. So, in this case, $m$ is odd, $k_{0}=\frac{m-1}{2}$ and
$(a,b)=(a,a+k_0) \stackrel{r}{\mapsto} (a+k_0+1,a) =(a+k_0+1,a+m)
=(a+k_0+1,a+k_0+1+k_0)$. It follows that
$$  O_{(a,b)} =\{ (0,k_0) , (1,k_0+1), \ldots ,  (m-1,m-1+k_0)\}$$
and
  $$|O_{(a,a+\frac{m-1}{2})}|=|O_{(a,b)}|=|O_{(a,a+k_0)}|=m.$$

Hence, if $m$ is even then all $r_{j}$-orbits in $X_{\sigma,j}^2$
have cardinality $2m$ and thus there are $\frac{m}{2}$ orbits in
this case. On the other hand, if $m$ is odd then all the orbits have
cardinality $2m$, except
$$O_{(a,\frac{m-1}{2}+a)} = \{ (b,b+\frac{m-1}{2})\mid b\in \Z/{m}\}$$
that has cardinality $m$. Thus, if $m$ is odd then there are
$\frac{m+1}{2}$ orbits, and the claim follows.

Hence, the number of $r$-orbits in $X^2_{\sigma}$ is greater than or
equal to
    $$ \sum_{1\leq j \leq s,\; |X_{\sigma ,j}| \mbox{ even}}
          \frac{|X_{\sigma , j}|}{2}
 + \sum_{1\leq j \leq s, \; |X_{\sigma ,j}| \mbox{ odd}}
 \frac{|X_{\sigma, j}| +1}{2}
 + \; |X_{\sigma, 0}| \geq \frac{|X_{\sigma}|}{2}.$$
Furthermore, if $|X_{\sigma}|$ is even then $\frac{|X_{\sigma}|}{2}
$ is a lower bound, otherwise $\frac{|X_{\sigma}|+1}{2} $ is a lower
bound. The equality only holds if either $|X_{\sigma}|$ is even and
$\sigma$ is a cycle of  length $|X_{\sigma}|$, or $|X_{\sigma}|$ is
odd and $\sigma$ is a cycle of length $|X_{\sigma}|$.

The number of $r$-orbits in $X^2$ is greater than or equal to
 $$ \sum_{\sigma \in \mathcal{S},\; |X_{\sigma}| \mbox{ even}}   \frac{|X_{\sigma}|}{2}+ \sum_{\sigma \in \mathcal{S},\; |X_{\sigma}| \mbox{ odd}}
   \frac{|X_{\sigma}|+1}{2}\geq \left\{
      \begin{array}{ll}
     \frac{|X|}{2} & \mbox{ if } |X| \mbox{ is even} \vspace{5pt}\\
     \vspace{5pt}
     \frac{|X|+1}{2} & \mbox{ if } |X| \mbox{ is odd}
    \end{array}  \right.
  $$
and the equality holds precisely when $|\mathcal{S}|$$= 1$, $\sigma
\in \mathcal{S}$ is a cycle of length $|X|$.
\end{proof}

A concrete example that satisfies the lower bounds in
Theorem~\ref{GeneralMinDim} is
$$r: \Z/(n) \rightarrow \Z/(n): (a,b) \mapsto (b+1,a).$$
If $n$ even then $\dim (A(K,\Z/(n),r)_{2}) =\frac{n}{2}$ and if $n$
is odd then $\dim (A(K,\Z/(n),r)_{2}) =\frac{n+1}{2}$.

Note that, in this case, $r_1^2(a,b,c)=(a+1,b+1,c)$. Hence
$(a+c-b,c,c)\in O_{(a,b,c)}$. Now
$r_2^{2(a-b)}(a+c-b,c,c)=(a+c-b,a+c-b,a+c-b)\in O_{(a,b,c)}$. Since
$r_2r_1^2(x,x,x)=(x+1,x+1,x+1)$, we have that $(0,0,0)\in
O_{(a,b,c)}$. Therefore $(\Z/(n))^3=O_{(0,0,0)}$. Thus
$$\dim (A(K,\Z/(n),r)_{3}) =1=\dim (A(K,\Z/(n),r)_{m}),$$
for all $m\geq 3$. Hence $\GK(A(K,\Z/(n),r))=1$.

Given a non-degenerate braided set $(X,r)$, Soloviev in
\cite{Soloviev} introduced its derived solution $(X,r')$ as follows.
Write $r(x,y)=(\sigma_x (y), \gamma_y (x))$ for all $x,y\in X$.
Then
 $$r'(x,y) =(\gamma_x (\sigma_{\gamma^{-1}_{y}(x)}(y)), x)$$
for all $x,y\in X$. Note that structure algebras
$A=A(K,X,r)=\oplus_{m\geq 0} A_m$ and $A'=A(K,X,r')=\oplus_{m\geq
0}A'_{m}$ are naturally  positively graded algebras (with $x\in X$
of degree 1). Because of Proposition~1.4 in \cite{JKV} there
exists a bijective $1$-cocycle $\phi: M(X,r) \rightarrow M(X,r')$
such that $\phi (x) =x$ and $\phi$ is degree preserving. In
particular, for any field $K$,
   $$\dim (A_m) =\dim (A'_{m}).$$ Hence, the number of
$r'$-orbits in $X^2$ is the same  as the number of $r$-orbits in
$X^2$. The following is now an immediate application of
Theorem~\ref{GeneralMinDim}.

\begin{corollary}\label{NDBS}
Let $(X,r)$ be a finite non-degenerate braided set. Let $K$ be a
field and $A=A(K,X,r)$, the structure algebra. Then, $\dim (A_{2})
\geq \frac{|X|}{2}$.
\end{corollary}

\begin{example}
Let $(X,r)$ be the quadratic set, with $X=\Z/(n)$, for an integer
$n>1$, and $r(x,y)=(y-1,x+2)$ for all $x,y\in X$. Then $(X,r)$ is a
non-degenerate braided set, its derived solution is $(X,r')$, with
$r'(x,y)=(y+1,x)$, for all $x,y\in X$, that we have studied just
before Corollary~\ref{NDBS}.  Thus, $\dim(A(K,X,r)_m)=\dim(A(K,X,r')_m)=1$
for all $m>2$, so that
$\GK(A(K,X,r))=\GK(A(K,X,r'))=1$ for any field $K$. Furthermore
$\dim(A(K,X,r)_2)=n/2$ if $n$ is even, and
$\dim(A(K,X,r)_2)=(n+1)/2$ if $n$ is odd.
\end{example}

\section{Some examples: the dihedral quandles} \label{quandles}
In this section we study the braided sets associated to the
dihedral quandles. These braided sets will be crucial for the main
result of this paper (proved in Section~\ref{sf-braided}).

Note that if $(X,r)$ is a finite $2$-cancellative square-free
non-degenerate braided set $(X,r)$ of cardinality $n$, then the
cardinality of  each $r$-orbit  in $X^2$  is at most $n$. Since
$n^2=n+(n-1)n$ and $X^2$ has at least $n$ orbits of cardinality $1$,
we have that $X^2$ has at least $2n-1$ orbits.

The following example is due to Leandro Vendramin (see
\cite[Remark~4.13]{GI2018}).

\begin{example} \label{dihedral}
Let $K$ be a field. Let $p$ be a prime number. Then there exists a
$2$-cancellative, non-degenerate, square-free braided set $(X,r)$
such that $|X|=p$ and its associated $K$-algebra $A=A(K,X,r)$
satisfies $\dim(A_2)=2p-1$.
\end{example}

\begin{proof}
Consider the dihedral group $D_{2p}=\langle g,s\mid g^p=s^2=1,\;
sgs=g^{-1}\rangle$. Let $X=\{ g^ks\mid 0\leq k<p\}$ and let $r\colon
X\times X\longrightarrow X\times X$ be the map defined by
$r(a,b)=(aba^{-1},a)$, for all $a,b\in X$. Clearly $(X,r)$ is a
non-degenerate square-free quadratic set. Note that
$$r(g^ks,g^ts)=(g^ksg^tssg^{-k},g^ks)=(g^{2k-t}s,g^ks),$$
$$r^2(g^ks,g^ts)=(g^{3k-2t}s,g^{2k-t}s),$$
and by induction it is easy to see that  for $m\geq 1$ we have
$$r^m(g^ks,g^ts)=(g^{(m+1)k-mt}s,g^{mk-(m-1)t}s).$$
Let $n$ be the smallest positive integer such that
$$(g^ks,g^ts)=(g^{(n+1)k-nt}s,g^{nk-(n-1)t}s).$$
Then $n(k-t)\equiv 0\; (\Mod p)$. Thus either $k\equiv t\; (\Mod p)$
or $n=p$. Hence $(X,r)$ is $2$-cancellative and $\dim(A_2)=2p-1$.
Note that
$$r_1r_2r_1(a,b,c)\;=\;r_1r_2(aba^{-1},a,c)\;=\;r_1(aba^{-1},aca^{-1},a)\;=\;(abcb^{-1}a^{-1},aba^{-1},a)$$ and
$$r_2r_1r_2(a,b,c)\;=\;r_2r_1(a,bcb^{-1},b)\;=\;r_2(abcb^{-1}a^{-1},a,b)\;=\;(abcb^{-1}a^{-1},aba^{-1},a).$$
Hence the result follows.
\end{proof}

The example constructed above is the braided set corresponding to
the dihedral quandle of order $p$.  Note that it is isomorphic to
the braided set $(\Z/(p),r)$, where $p$ is prime and $r$ is
defined by $r(k,s)=(2k-s, k)$, for all $k,s\in \Z/(p)$.

\begin{proposition}
Let $p$ be an odd prime. Let $(X,r)$ be the braided set
corresponding to the dihedral quandle of order $p$. Let $K$ be a
field. Let $A_n=Span_K\{y_1\cdots y_n\mid y_1,\dots ,y_n\in X\}$
in $A(K,X,r)$. Then $\dim(A_n)=2p$, for all $n\geq 3$, and thus
the $GK$-dimension of $A(K,X,r)$ is $\GK(A(K,X,r))=1$.
\end{proposition}
\begin{proof}
We may assume that $X=\{ x_i\mid i\in \Z/(p)\}$,
$r(x_i,x_j)=(x_{\sigma_i(j)},x_i)$, for all $i,j\in \Z/(p)$ and
$\sigma_{i}(j)=2i-j$. Hence in $M(X,r)$ we have that
$$x_{i}x_{j}=x_{2i-j}x_i,$$
for all $i,j\in\Z/(p)$. Hence $x_{i}^2x_j=x_jx_i^2$, for all $i,j\in
\Z/(p)$. Let $i,j\in\Z/(p)$ be distinct elements. Note that
$$x_ix_j=x_{2i-j}x_i=x_{3i-2j}x_{2i-j}=\dots =x_{(t+1)i-tj}x_{ti-(t-1)j},$$
for all $t\in\Z/(p)$. Since $i\neq j$ and $p$ is a prime, there
exists $t\in \Z/(p)$ such that $(t+1)i-tj=0$, thus $ti-(t-1)j=j-i$.
Hence
\begin{equation}\label{eqGK1}
x_ix_j=x_0x_{j-i}.
\end{equation}

Let $n$ be an integer $n\geq 3$. We shall prove by induction on $n$
that
\begin{equation}\label{basis}(x_0^n,\dots ,x_{p-1}^n,
x_0^{n-1}x_1,\dots ,x_0^{n-1}x_{p-1}, x_0^{n-2}x_1^2)\end{equation}
is a basis of $A_n$, and therefore $\dim(A_n)=2p$.

First consider the map $f\colon X\longrightarrow D_{2p}$ defined
by $f(x_i)=g^{i}s$. Since $g^{i}sg^{j}s=g^{i-j}=g^{2i-j}sg^{i}s$,
$f$ extends to a homomorphism $f'\colon M(X,r)\longrightarrow
D_{2p}$. Note that $f'(x_0^{n-2}x_1^2)=s^{n-2}$ and
$f'(x_0^{n-1}x_i)=s^{n-1}g^{i}s$. Therefore the $2p$ elements
$x_i^n$, $x_0^{n-2}x_1^2$ and $x_0^{n-1}x_j$, for $i,j\in\Z/(p)$
with $j\neq 0$, are distinct.

For $n=3$, by (\ref{eqGK1}), we have that $$\{ x_i^3\mid i\in
\Z/(p)\}\cup\{ x_0x_i^2\mid
i\in\Z/(p)\setminus\{0\}\}\cup\{x_0^2x_i\mid
i\in\Z/(p)\setminus\{0\}\}$$ generates $A_3$ as a vector space. By
(\ref{eqGK1}) and the defining relations of $M(X,r)$, for $i\neq
0$ we have
$$x_0x_i^2=x_1x_{i+1}x_i=x_1x_0x_{-1}=x_1(x_1x_0)=x_0x_1^2,$$
where the last equality follows because $x_1^2$ is a central
element of $M(X,r)$. Hence
$$\{ x_i^3\mid i\in
\Z/(p)\}\cup\{ x_0x_1^2\}\cup\{x_0^2x_i\mid
i\in\Z/(p)\setminus\{0\}\}$$ generates $A_3$.  Thus (\ref{basis}) is
a basis of $A_n$ for $n=3$.

Suppose that (\ref{basis}) is a basis of $A_n$ for some $n\geq 3$.
Let $x_{i_1},\dots ,x_{i_{n+1}}\in X$. Then either
$x_{i_j}=x_{i_k}$, for all $j,k$, or there exists $0\leq j<n+1$ such
that $x_{i_j}\neq x_{i_{j+1}}$. In the second case, by (\ref{eqGK1})
and the defining relations, we have that
$$x_{i_1}x_{i_2}\cdots x_{i_{n+1}}
x_{i_1}\cdots x_{i_{j-1}}(x_0x_{i_{j+1}-i_{j}})x_{i_{j+2}}\cdots
x_{i_{n+1}} =x_0x_{-i_1}\cdots
x_{-i_{j-1}}x_{i_{j+1}-i_{j}}x_{i_{j+2}}\cdots x_{i_{n+1}}.
$$
Hence, by the induction hypothesis
$$\{ x_i^{n+1}\mid i\in
\Z/(p)\}\cup\{ x_0x_i^n\mid
i\in\Z/(p)\setminus\{0\}\}\cup\{x_0^{n}x_i\mid
i\in\Z/(p)\setminus\{0\}\}\cup\{x_0^{n-1}x_1^2\}$$ generates
$A_{n+1}$. We claim that if $i\neq 0$, then
$$x_0x_i^n=\left\{\begin{array}{ll}
x_0^nx_i&\quad\mbox{if $n$ is odd}\\
x_0^{n-1}x_1^2&\quad\mbox{if $n$ is even}\end{array}\right.$$
Suppose that $n$ is odd. In this case, since
$x_0x_i^2=x_0x_1^2=x_1^2x_0$, we have that
$$x_0x_i^n=x_0x_1^{n-1}x_i.$$
Note also that $x_0x_1^2=x_{-1}^2x_0$. Hence, we also have that
$$x_0x_i^n=x_0x_{-1}^{n-1}x_i.$$
If $i\neq 1$, then, by (\ref{eqGK1}),
$$x_1^2x_i=x_1x_0x_{i-1}=x_0x_{-1}x_{i-1}=x_0^2x_i.$$
If $i=1$, then $i\neq -1$ and by (\ref{eqGK1}),
\begin{equation}\label{eqx1xi}
x_{-1}^2x_i=x_{-1}x_0x_{i+1}=x_0x_{1}x_{i+1}=x_0^2x_i.
\end{equation}
Therefore, since the elements $x_j^2$ are central in $M(X,r)$, we
obtain that
$$x_0x_i^n=x_0^{n}x_i.$$
Suppose that $n$ is even. In this case, since
$x_0x_i^2=x_0x_1^2=x_1^2x_0$, we have that
$$x_0x_i^n=x_0x_1^{n}.$$
Since $x_0x_1^2=x_{-1}^2x_0$, we also have that
$$x_0x_i^n=x_0x_{-1}^{n-2}x_1^2.$$
By (\ref{eqx1xi}),
$$x_0x_i^n=x_0^{n-1}x_1^2.$$
Hence the claim is proved. Therefore,
$$\{ x_i^{n+1}\mid i\in
\Z/(p)\}\cup\{x_0^{n}x_i\mid
i\in\Z/(p)\setminus\{0\}\}\cup\{x_0^{n-1}x_1^2\}$$ generates
$A_{n+1}$. Hence, we have proved by induction that  (\ref{basis}) is
a basis of $A_n$ for all $n\geq 3$ and thus $\dim(A_n)=2p$. By
Example~\ref{dihedral}, $\dim(A_2)=2p-1$. Hence, for every $n\geq
2$,
$$\dim(\oplus_{i=0}^nA_i)=1+p+2p-1+(n-2)2p=2pn-p.$$
Thus $\GK(A(K,X,r))=1$, and the result is proved.
\end{proof}

Note that, for $p=2$, the braided set corresponding to the
dihedral quandle of order $2$ is the trivial braided set $(X,r)$
of cardinality 2, that is $r(x,y)=(y,x)$, for all $x,y\in X$. In
this case $M(X,r)$ is the free abelian monoid of rank 2 and, for
every field $K$, $\GK(A(K,X,r))=2$.

 The following result completes our discussion of the dihedral
quandles and  is needed to prove Theorem~\ref{minimality}.

\begin{proposition}\label{quandle2}
Let $n$ be an integer greater than $2$. Let $(X,r)$ be the braided
set corresponding to the dihedral quandle of order $n$. Let $K$ be a
field. Then $\dim(A(K,X,r)_2)=2n-1$ if and only if $n$ is prime.
\end{proposition}

\begin{proof} We may assume that $X=\Z/(n)$,
$r(i,j)=(\sigma_i(j),i)$, for all $i,j\in \Z/(n)$ and
$\sigma_{i}(j)=2i-j$. Note that $\sigma_x\sigma_y(z)=z+2(x-y)$, for
all $x,y,z\in X$. By Corollary~\ref{rn},
\begin{eqnarray*}
r^{2k+1}(x,y)&=&((\sigma_x\sigma_y)^{k+1}(y),(\sigma_x\sigma_y)^{k}(x))\\
&=&(y+2(k+1)(x-y),x+2k(x-y)),\\
r^{2k}(x,y)&=&((\sigma_x\sigma_y)^{k}
(x),(\sigma_x\sigma_y)^{k}(y))\\
&=&(x+2k(x-y),y+2k(x-y)),
\end{eqnarray*}
for all non-negative integers $k$ and $x,y\in X$. Now it is easy to
see that $(X,r)$ is $2$-cancellative. Hence $\dim(A(K,X,r)_2)=2n-1$
if and only if there are $n$ $r$-orbits of cardinality $1$ and $n-1$
$r$-orbits of cardinality $n$ in $X^2$.

By Example~\ref{dihedral}, if $n$ is prime, then
$\dim(A(K,X,r)_2)=2n-1$.

So, for the remainder of the proof we may suppose that $n$ is not
prime.

First, suppose that $n$ is even. Thus $n=2t$ for some positive
integer $t$. Now $\sigma_t(0)=0$ and $\sigma_0(t)=t$. In this
case, the $r$-orbit of $(t,0)$ is $O_{(t,0)}=\{ (t,0),(0,t)\}$.
Therefore $\dim(A(K,X,r)_2)>2n-1$.

Second, suppose that $n$ is odd. Then there exist integers
$2<a,b<n$ such that $ab=n$. By Corollary~\ref{rn},
\begin{eqnarray*}
r^{2k}(a,0)&=&((\sigma_a\sigma_0)^{k}
(a),(\sigma_a\sigma_0)^{k}(0))\\
&=&(a+2ka,2ka)
\end{eqnarray*}
for all non-negative integers $k$. Hence the $r$-orbit of $(a,0)$
has cardinality $|O_{(a,0)}|\leq 2b<ab=n$. Thus, also in this case,
we have that $\dim(A(K,X,r)_2)>2n-1$.

Therefore, the result follows.
\end{proof}

\section{Square-free non-degenerate braided sets}  \label{sf-braided}
 In this section we prove the main result of this paper. We
begin by proving a lower bound for the dimension of $A(K,X,r)_2$
where $(X,r)$ is a finite non-degenerate square-free braided  set.
In the main results, Theorem~\ref{minimal} and
Corollary~\ref{cor-minimal}, we  characterise when  this lower bound
is achieved.

Note that if $(X,r)$ is a square-free quadratic set, then the
$r$-orbits of the elements $(x,x)$ in $X^2$ have cardinality $1$.
Therefore, if $X$ is finite of cardinality $n>1$, then the number of
orbits in $X^2$ is at least $n+1$. The following example shows that
this bound is attainable for finite non-degenerate square-free
quadratic sets.

\begin{example} \label{dim-An-sf} Let $K$ be a field. Let $n>1$ be an integer.
Then there exists a square-free non-degenerate quadratic set $(X,r)$
such that $|X|=n$ and  $\dim(A_m)=n+1$ for all $m>1$, and its
associated $K$-algebra $A=A(K,X,r)$ has $GK$-dimension $1$.
\end{example}

\begin{proof}
Let $X=\{ 1,2,\dots ,n\}$ and let $r\colon X\times X\longrightarrow
X\times X$ be the map defined by $r(x,y)=(\sigma_x(y),x)$ for all
$x,y\in X$, where $\sigma_{n-1}=\id_X$, $\sigma_n=(1,2,\dots ,n-1)$
and, for $i\in X\setminus\{ n-1,n\}$, the permutations
$\sigma_{i}\in \Sym_n$ are defined as follows:
$$\sigma_{i}=(i+1,i+2,\dots, n).$$
Note that
\begin{eqnarray*}&&(1,n)\overset{r}{\mapsto}(2,1)\overset{r}{\mapsto}
(1,2)\overset{r}{\mapsto} \cdots\overset{r}{\mapsto}
(n-1,1)\overset{r}{\mapsto} (1,n-1)\overset{r}{\mapsto}
(n,1)\\
&&\overset{r}{\mapsto} (2,n)\overset{r}{\mapsto}
(3,2)\overset{r}{\mapsto} (2,3)\overset{r}{\mapsto} \cdots
\overset{r}{\mapsto}(n-1,2)\overset{r}{\mapsto}
(2,n-1)\overset{r}{\mapsto} (n,2)\\
&&\overset{r}{\mapsto} (3,n)
\overset{r}{\mapsto} \cdots\\
&&\overset{r}{\mapsto}(n-1,n)\overset{r}{\mapsto}
(n,n-1)\overset{r}{\mapsto} (1,n).
\end{eqnarray*}
Hence $(X,r)$ is a square-free non-degenerate quadratic set. Let
$A=A(K,X,r)$.   By Remark~\ref{monoid}, the number of elements of
length $2$ in $M(X,r)$ is $n+1$, and thus $\dim(A_2)=n+1$. It is
easy to see that $\dim(A_m)=n+1$ for all $m\geq 2$. Hence $\dim_K
(\oplus_{i=0}^mA_i)=(n+1)m$, for every positive integer $m$.
Therefore $\GK(A)=1$ and this finishes the proof.
\end{proof}

The main results of this section are concerned with the
minimality condition for finite square-free non-degenerate braided
sets.

\begin{theorem} \label{SquareFreeMinDim}
Let $(X,r)$ be a finite square-free non-degenerate braided set. Then
the number  of $r$-orbits in $X^{2}$ is at least $2|X|-1$, that is
$\dim A_2 \geq 2 |X|-1$ where $A=A(K,X,r)$.
\end{theorem}
\begin{proof}
Note that the derived braided set $(X,r')$ of $(X,r)$ also is
square-free. Hence, as explained in  Section~\ref{sect-non-deg},
without loss of generality, we may assume that $(X,r)$ is $\SD$,
that is $r(a,b)=(\sigma_a (b), a)$, for all $a,b\in X$ and $\sigma_a
\in \Sym (X)$. We use the notation introduced in the proof of
Theorem~\ref{GeneralMinDim}. So $\mathcal{S} =\{ \sigma_a \mid a\in
X\}$ and $X_{\sigma}=\{ a\in X \mid \sigma_a =\sigma \}$, for
$\sigma \in \mathcal{S}$. Hence, $X=\bigcup_{\sigma\in \mathcal{S}}
X_{\sigma}$, a disjoint union. Since, by assumption, $(X,r)$ is
square-free,  for every $\sigma \in \mathcal{S}$ and $a,b\in
X_{\sigma}$, we have that
 $$r(a,b)=(\sigma_a (b),a) =(\sigma (b),a)=(\sigma_{b}(b),a)=(b,a).$$
Thus, the restriction of $r$ to $X_{\sigma}^{2}$ is the trivial
solution. Therefore, for $a,b\in X_{\sigma}$, the orbit
$O_{(a,b)}=\{ (a,b),(b,a)\}$. Hence, there are $|X_{\sigma}| +
{|X_{\sigma}| \choose 2}$ $r$-orbits in $X_{\sigma}^2$.  (Note
that we agree, as is common, that ${1 \choose 2} =0$.)

Let $\sigma , \tau, \tau'\in \mathcal{S}$ be three distinct
elements. Let $a,a'\in X_{\sigma}$, $b\in X_{\tau}$ and $c\in
X_{\tau'}$. We claim that  $O_{(a,b)} \neq  O_{(a',c)}$. We prove
this by contradiction. So, suppose that $O_{(a,b)} = O_{(a',c)}$.
Then, in $M(X,r)$, we have that $ab=a'c$.  By the presentation of
$M(X,r)$ and Lemma~\ref{conditions}, we know that there is a
unique semigroup homomorphism $\varphi : M(X,r) \rightarrow
\mathcal{G}(X,r)=\langle \sigma_a \mid a\in X \rangle$ such that
$\varphi (a) =\sigma_a$ for all $a\in X$. Hence, $\sigma_a
\sigma_b = \sigma_{a'}\sigma_c$. So, $\sigma \tau =\sigma \tau'$,
a contradiction. Hence, the claim is proven.

Therefore, if $a_{\sigma}\in X_{\sigma}$ for $\sigma \in
\mathcal{S}$, then the sets $O_{(a_{\sigma},a_{\tau})}$ are all
different for $\tau \in \mathcal{S}\setminus \{ \sigma \}$. Thus the
number  of $r$-orbits in $X^{2}$ is greater than or equal to
  $$t =|\mathcal{S}| -1 +\sum_{\sigma \in \mathcal{S}}
  \left( |X_{\sigma}| + {|X_{\sigma}| \choose 2}\right) .$$
Since $ |X_{\sigma}| + {|X_{\sigma}| \choose 2} \geq 2|X_{\sigma}|
-1$ we get that $t\geq |\mathcal{S}| -1 +\sum_{\sigma \in
\mathcal{S}}
  \left(2 |X_{\sigma}| -1\right) =2|X|-1$.
\end{proof}

Note that $ |X_{\sigma}| + {|X_{\sigma}| \choose 2} = 2|X_{\sigma}|
-1$ if and only if $1\leq |X_{\sigma}|\leq 2$. Thus, if the number
of $r$-orbits in $X^{2}$ is $2|X|-1$ then $|X_{\sigma}|\leq 2$ for
all $\sigma \in \mathcal{S}$.

Note that the trivial braided set $(X,r)$ with $X=\{1,2\}$ and
$r(x,y)=(y,x)$ for all $x,y\in X$, is a square-free non-degenerate
$\SD$ braided set, such that $\sigma_1=\sigma_2=\id$ and the
number of $r$-orbits in $X^2$ is $3=2|X|-1$.

Another example of a finite square-free non-degenerate
 $\SD$ braided set $(X,r)$ such that $|X_{\sigma}|=2$ for some
$\sigma \in \mathcal{S}$ and the number of $r$-orbits in $X^2$ is
$2|X|-1$ is the following. Let $X=\{ 1,2,3\}$ and
$r(x,y)=(\sigma_x(y),x)$, for all $x,y\in X$, where
$\sigma_1=\sigma_2=\id$ and $\sigma_3=(1,2)$. In this case, the
$r$-orbits in $X^2$ are:
$$\{ (1,1)\},\; \{ (2,2)\},\;\{ (3,3)\},\;\{ (1,2), (2,1)\},
\;\mbox{ and }\;\{ (1,3), (3,1), (2,3), (3,2)\}.$$

Let $(X,r)$ be a finite non-degenerate square-free braided set.
Hence, as before, to study the number of $r$-orbits in $X^2$, we may
assume that $(X,r)$ is $\SD$, that is $r(x,y)=(\sigma_x(y),x)$, for
all $x,y\in X$, where $\sigma_x\in \Sym_X$ satisfies
$\sigma_x(x)=x$, for all $x\in X$.

In the next result we deal with the special case that each
$|X_{\sigma}|=1$, where  $X_{\sigma}=\{ a\in X \mid \sigma_a =\sigma
\}$ and $X=\bigcup_{\sigma\in \mathcal{S}} X_{\sigma}$.

\begin{theorem}\label{minimality}
Let $(X,r)$ be a finite non-degenerate square-free $\SD$ braided
set. Suppose that $|X|>1$ and that the number of $r$-orbits in $X^2$
is $2|X|-1$. Suppose that $r(x,y)=(\sigma_x(y),x)$, for all $x,y\in
X$, and that $\sigma_x\neq \sigma_y$ if $x\neq y$. Then $|X|$ is an
odd prime and $(X,r)$ is the braided set associated to the dihedral
quandle.
\end{theorem}

\begin{proof}
Note that if $|X|=2$, then, since $(X,r)$ is non-degenerate and
square-free,  $\sigma_x=\id$ for all $x\in X$, thus we get a
contradiction. Hence $|X|>2$.

We use the notation introduced in the proof of
Theorem~\ref{GeneralMinDim}. So $\mathcal{S} =\{ \sigma_a \mid a\in
X\}$ and $X_{\sigma}=\{ a\in X \mid \sigma_a =\sigma \}$, for
$\sigma \in \mathcal{S}$.

By the hypothesis we have that $|X_{\sigma}|=1$ for all $\sigma \in
\mathcal{S}$. Let $(x,y),(z,t)\in X^2$ be two elements in the same
$r$-orbit. By Lemma~\ref{conditions},
$\sigma_x\sigma_y=\sigma_z\sigma_t$. Since $|X_{\sigma}|=1$ for all
$\sigma \in \mathcal{S}$, we have that $x\neq z$ if and only if
$y\neq t$. Therefore $(X,r)$ is $2$-cancellative. Hence
$|O_{(x,y)}|\leq |X|$. Note that if $x,y,z\in X$ are 3 distinct
elements, then $O_{(x,y)}\neq O_{(x,z)}$. Since the number of
$r$-orbits in $X^2$ is $2|X|-1$ and there are $|X|$ singleton orbits
of the form $\{ (x,x)\}$, we have that the other $|X|-1$ orbits are
$O_{(x,y)}$ for some fixed element $x\in X$ and for all $y\in
X\setminus\{ x\}$. Furthermore, the cardinality of every
non-singleton $r$-orbit $O_{(x,y)}$ in $X^2$ is $|X|$.

Hence every non-singleton $r$-orbit in $X^2$ is of the form
\begin{eqnarray}\label{chain}
(a_2,a_1)\overset{r}{\mapsto}(a_3,a_2)\overset{r}{\mapsto}(a_4,a_3)
\overset{r}{\mapsto}\cdots\overset{r}{\mapsto}(a_n,a_{n-1})\overset{r}{\mapsto}(a_1,a_n)
\overset{r}{\mapsto}(a_2,a_1),\end{eqnarray} where
$n=|\{(a_{i+1},a_i)\mid i\in\Z/(n) \}|=|X|$ and $X=\{ a_1,a_2,\dots
, a_n\}$.

By Corollary~\ref{rn},
\begin{eqnarray} \label{rmaps}
r^{2k-1}(a_2,a_1)&=&((\sigma_{a_2}\sigma_{a_1})^{k}(a_{1}),(\sigma_{a_2}\sigma_{a_1})^{k-1}(a_2)),\\
r^{2k}(a_2,a_1)&=&((\sigma_{a_2}\sigma_{a_1})^{k}(a_2),(\sigma_{a_2}\sigma_{a_1})^{k}(a_1)),
\nonumber
\end{eqnarray}
for all positive integers $k$. Let
$\sigma=\sigma_{a_2}\sigma_{a_1}$. Since $(X,r)$ is square-free,
$\sigma_x(x)=x$ for all $x\in X$. Suppose that $\sigma_x(y)=y$,
for some $y\neq x$. By Lemma~\ref{conditions},
$\sigma_x\sigma_y=\sigma_y\sigma_x$. Then $O_{(x,y)}=\{
(x,y),(y,x)\}$, because
$\sigma_{\sigma_y(x)}=\sigma_y\sigma_x\sigma_{y}^{-1}=\sigma_x$
and thus $\sigma_y(x)=x$. But this means that $|X|=2$, a
contradiction. Hence $x$ is the only fixed point of $\sigma_x$.

By Lemma~\ref{inductive},
\begin{equation}
\label{conjugeven}\sigma_{a_{2k+1}}=\sigma^{k}\sigma_{a_1}\sigma^{-k},\quad
 \sigma_{a_{2k}}=\sigma^{k-1}\sigma_{a_2}\sigma^{-k+1},
\end{equation}
for all $k\in \Z/(n)$. In particular $\sigma_{a_1}$ and
$\sigma_{a_3}$ are in the same conjugacy class. Now consider the
$r$-orbit $O_{(a_3,a_1)}$ of $(a_3,a_1)$ in $X^2$. Let $b_1=a_1$ and
$b_2=a_3$. Then the orbit $O_{(a_3,a_1)}=O_{(b_2,b_1)}$ is of the
form (\ref{chain}) changing the $a$'s by $b$'s
\begin{eqnarray*}
(b_2,b_1)\overset{r}{\mapsto}(b_3,b_2)\overset{r}{\mapsto}(b_4,b_3)
\overset{r}{\mapsto}\cdots\overset{r}{\mapsto}(b_n,b_{n-1})\overset{r}{\mapsto}(b_1,b_n)
\overset{r}{\mapsto}(b_2,b_1),\end{eqnarray*} and $X=\{
b_1,b_2,\dots, b_n\}$. Hence  for every $x\in X$, $\sigma_x$ is
either a conjugate of $\sigma_{a_1}$ or a conjugate of
$\sigma_{a_3}$. Therefore, all the $\sigma_{a_j}$ belong to the
same conjugacy class in $\mathcal{G}(X,r)=\langle \sigma_x\mid
x\in X\rangle$. In fact, since
$\sigma_{\sigma_x(y)}=\sigma_x\sigma_y\sigma_x^{-1}$, for all
$x,y\in X$, we have that $\mathcal{S}=\{\sigma_{x}\mid x\in X\}$
is a conjugacy class in $\mathcal{G}(X,r)$ and has cardinality
$n$. By (\ref{rmaps}) applied to $O_{(a_2,a_1)}$ and
$O_{(b_2,b_1)}$, it is clear that $\mathcal{G}(X,r)$ acts
transitively on $X$.

Note also that,
$$\mathcal{G}(X,r)=\langle \sigma_{x},\sigma_{y}\rangle=\langle \sigma_{x},\sigma_y\sigma_x\rangle,$$
for distinct $x,y\in X$. Let $G=\mathcal{G}(X,r)$. Let $Z(G)$ denote
the center of $G$. It follows that $Z(G) =C_{G}(\sigma_{x})\cap
C_{G}(\sigma_{y})$, for every $x\neq y$. Let $\eta\in Z(G)$. Since
$\eta\in C_{G}(\sigma_x)$, we have that
$\eta(x)=\eta(\sigma_x(x))=\sigma_x(\eta(x))$, for all $x\in X$.
Since $x$ is the only fixed point of $\sigma_x$, it follows that
$\eta(x)=x$ for all $x\in X$. Therefore,
\begin{equation}\label{two-cent}
C_{G}(\sigma_{x})\cap C_{G}(\sigma_{y})=Z(G) =1 \quad \mbox{ for
every } x\neq y.
\end{equation}

The argument given above also implies that the subgroup
$H_{a_1}=\{\tau \in G \mid \tau(a_1) =a_1\}$ of $G$ satisfies
$C_{G}(\sigma_{a_1})\subseteq H_{a_1}$. On the other hand, we know
that $[G:C_{G}(\sigma_{a_1})]=n$. Since $G$ acts transitively on
$X$, $[G:H_{a_1}]=n$. It follows that $H_{a_1}=C_{G}(\sigma_{a_1})$.

Define $Y = \{ \tau \in G \mid \tau (x)\neq x \mbox{ for every }
x\in X \}$. Now, we have
\begin{equation*} \label{union}
G=C_{G}(\sigma_{a_1})\cup \cdots \cup C_{G}(\sigma_{a_n})\cup Y,
\end{equation*}
with $Y\cap C_{G}(\sigma_{a_i})=\emptyset$, for all $i=1,\dots
,n$. We also know that $H_{x}= C_{G}(\sigma_{x})$ for every $x$. In
view of (\ref{two-cent}) it follows that
$$|G|=n|C_{G}(\sigma_1)| = n(|C_{G}(\sigma_1)| -1) +1 + |Y|.$$
So $|Y|=n-1$.

If $n$ is odd,  $n=2t+1$, for some $t$, and by (\ref{chain}) and
(\ref{rmaps}) we get
$$\sigma=(a_1,a_3,\dots a_{n},a_2,a_4,\dots ,a_{2t}).$$

If $n$ is even,  $n=2t$, for some $t$, and by (\ref{chain}) and
(\ref{rmaps}) we get
$$\sigma=(a_1,a_3,\dots a_{n-1})(a_2,a_4,\dots ,a_{2t}).$$
Hence, in both cases, $\sigma\in Y$ (as it does not have fixed
points). Considering the $r$-orbits $O_{(x,y)}$ in $X^2$, with
$x\neq y$, we have that $\sigma_x\sigma_y,
\sigma_y^{-1}\sigma_x^{-1}\in Y$. Hence $Y=\{
\sigma_{a_i}\sigma_{a_1}\mid i\neq 1\}$. Therefore, for every $i\neq
1 $ there exists $j\neq 1$ such that $\sigma_{a_i}^{-1}
\sigma_{a_1}^{-1} = \sigma_{a_j} \sigma_{a_1}$. Thus,
$\sigma_{a_i}\sigma_{a_j} = \sigma_{a_1}^{-2}$. Since $a_1$ is a
fixed point of $\sigma_{a_1}$, we have that $i=j$ and
$\sigma_x^{2}=\id$ for all $x\in X$. Hence $\sigma_x$ is a product
of disjoint transpositions and has a unique fixed point. Therefore
$|X|$ is odd.

But $\sigma $ is of order $n$ and $\sigma \in Y$. Therefore $Y\cup
\{1\} = \langle \sigma \rangle$. Clearly, this implies that $\langle
\sigma \rangle$ is a normal subgroup of $G$. So, $G=\langle
\sigma_{a_1},\sigma\rangle$ is a semidirect product
\begin{equation} \label{semidirect} G=\langle \sigma \rangle \rtimes
\langle\sigma_{a_1}\rangle.
\end{equation}
Since $\sigma\sigma_{a_1}\neq\sigma_{a_1}\sigma$, it is clear that
$G$ is the dihedral group of order $2n$.  Let $\sigma(X)=\{
\sigma_{x}\mid x\in X\}\subset G$. Note that the map $f\colon
X\longrightarrow \sigma(X)$ is an isomorphism of the braided sets
$(X,r)$ and $(\sigma(X),s)$, where
$s(\sigma_x,\sigma_y)=(\sigma_x\sigma_y\sigma_x^{-1},\sigma_x)$, for
all $x,y\in X$. Hence $(X,r)$ is the braided set associated to the
dihedral quandle of cardinality $n$. Therefore, by
Proposition~\ref{quandle2}, $n$ is prime and the result follows.
\end{proof}

The following technical lemma is another intermediate step that is
needed to prove the main result of this section. We study the
number of $s$-orbits of a non-degenerate square-free $\SD$ braided
set $(Y,s)$ such that $Y_{\sigma}=\{ a\in Y\mid \sigma_a=\sigma\}$
has cardinality $2$ for all $\sigma\in \mathcal{S}=\{\sigma_y\mid
y\in Y\}$ and $\Ret(Y,s)$ is the braided set associated to the
dihedral quandle of odd prime cardinality.

\begin{lemma}\label{extension}
Let $p$ be an odd prime. Let $X=\Z/(p)$ and let $r\colon
X^2\longrightarrow  X^2$ be the map defined by
$r(i,j)=(\sigma_i(j),i)$, where $\sigma _i(j)=2i-j$, for all $i,j\in
X$. We know that $(X,r)$ is the braided set associated to the
dihedral quandle. Let $Y=X\times \{ 1,2\}$. Let $s\colon
Y^2\longrightarrow Y^2$ be a map such that:
\begin{itemize}
\item[(1)] $s((i,m_1),(j,m_2))=(\sigma_{(i,m_1)}(j,m_2),(i,m_1))$, for all $(i,m_1),(j,m_2)\in Y$.
\item[(2)]$(Y,s)$ is a non-degenerate square-free $\SD$
braided set.
\item[(3)] $\sigma_{(i,1)}=\sigma_{(i,2)}$, for all $i,j\in X$.
\item[(4)] The natural projection $Y\longrightarrow X$ is a
homomorphism of braided sets from $(Y,s)$ to $(X,r)$.
\end{itemize}
Then the number of $s$-orbits in $Y^2$ is greater than $2|Y|-1$.
\end{lemma}
\begin{proof}
By (3) and (4), note that
$\sigma_{(i,m_1)}(j,m_2)=(\sigma_i(j),\alpha_{(i,j)}(m_2))$, for
some $\alpha_{(i,j)}\colon \{1,2\}\longrightarrow \{ 1,2\}$, for all
$i,j\in X$ and $m_1,m_2\in\{ 1,2\}$. Since $(Y,s)$ is
non-degenerate, $\sigma_{(i,m_1)}(j,1)\neq \sigma_{(i,m_1)}(j,2)$.
Hence $\alpha_{(i,j)}(1)\neq\alpha_{(i,j)}(2)$ and thus
$\alpha_{(i,j)}\in\Sym_2$. Since $(Y,s)$ is square-free,
$\alpha_{(i,i)}=\id$, for all $i\in X$. By Lemma~\ref{conditions},
$\sigma_{(i,m_1)}\sigma_{(j,m_2)}=\sigma_{(\sigma_i(j),\alpha_{(i,j)}(m_2))}\sigma_{(i,m_1)}$.
Hence
$$\sigma_{(i,m_1)}\sigma_{(j,m_2)}(k,m_3)\;=\;\sigma_{(\sigma_i(j),\alpha_{(i,j)}(m_2))}\sigma_{(i,m_1)}(k,m_3).$$
Since
$$\sigma_{(i,m_1)}\sigma_{(j,m_2)}(k,m_3)
\;=\;\sigma_{(i,m_1)}(\sigma_j(k),\alpha_{(j,k)}(m_3)) \;=\;
(\sigma_i(\sigma_j(k)),\alpha_{(i,\sigma_j(k))}\alpha_{(j,k)}(m_3))
$$
and
\begin{eqnarray*}
\sigma_{(\sigma_i(j),\alpha_{(i,j)}(m_2))}\sigma_{(i,m_1)}(k,m_3)&=&\sigma_{(\sigma_i(j),\alpha_{(i,j)}(m_2))}(\sigma_i(k),\alpha_{(i,k)}(m_3))\\
&=&(\sigma_{\sigma_i(j)}\sigma_i(k),\alpha_{(\sigma_i(j),\sigma_i(k))}\alpha_{(i,k)}(m_3)),
\end{eqnarray*}
we have that
\begin{equation}\label{alpha}\alpha_{(i,\sigma_j(k))}\alpha_{(j,k)}=\alpha_{(\sigma_i(j),\sigma_i(k))}\alpha_{(i,k)},
\end{equation}
for all $i,j,k\in X$. For $i=\sigma_j(k)$ and since
$\alpha_{(\sigma_j(k),\sigma_j(k))}=\id$, we get
\begin{equation}\label{alpha2}\alpha_{(j,k)}=\alpha_{(\sigma_{\sigma_j(k)}(j),\sigma_{\sigma_{j}(k)}(k))}\alpha_{(\sigma_j(k),k)},
\end{equation}
for all $j,k\in X$. Hence, by (\ref{alpha2}) and the definition of
$\sigma_x$, we have
\begin{eqnarray*}
\alpha_{(j,k)}
&=&\alpha_{(\sigma_{\sigma_j(k)}(j),\sigma_{\sigma_{j}(k)}(k))}\alpha_{(\sigma_j(k),k)}\\
&=&\alpha_{(\sigma_{2j-k}(j),\sigma_{2j-k}(k))}\alpha_{(2j-k,k)}\\
&=&\alpha_{(3j-2k,4j-3k)}\alpha_{(2j-k,k)}\\
&=&\alpha_{(\sigma_{\sigma_{3j-2k}(4j-3k)}(3j-2k),\sigma_{\sigma_{3j-2k}(4j-3k)}(4j-3k))}\alpha_{(\sigma_{3j-2k}(4j-3k),4j-3k)}\alpha_{(2j-k,k)}\\
&=&\alpha_{(\sigma_{2j-k}(3j-2k),\sigma_{2j-k}(4j-3k))}\alpha_{(2j-k,4j-3k)}\alpha_{(2j-k,k)}\\
&=&\alpha_{(j,k)}\alpha_{(2j-k,4j-3k)}\alpha_{(2j-k,k)},
\end{eqnarray*}
and thus $\alpha_{(2j-k,4j-3k)}=\alpha_{(2j-k,k)}$, for all $j,k\in
X$. Therefore
\begin{equation}\label{alpha3}
\alpha_{(t,2t-k)}=\alpha_{(t,k)},
\end{equation}
for all $t,k\in X$.
\bigskip

{\em Claim:} We shall prove that
$$\prod_{y\in X}\alpha_{(x+y,y)}=\id,$$
for all $x\in X$.
\bigskip

By (\ref{alpha2}) and the definition of $\sigma_i$, we have
\begin{eqnarray*}\prod_{y\in X}\alpha_{(\frac{x}{2}+y,y)}&=&\prod_{y\in X}\alpha_{(\sigma_{\sigma_{\frac{x}{2}+y}(y)}(\frac{x}{2}+y),\sigma_{\sigma_{\frac{x}{2}+y}(y)}(y))}
\alpha_{(\sigma_{\frac{x}{2}+y}(y),y)}\\
&=&\prod_{y\in
X}\alpha_{(\sigma_{x+y}(\frac{x}{2}+y),\sigma_{x+y}(y))}
\alpha_{(x+y,y)}\\
&=&\prod_{y\in X}\alpha_{(\frac{3x}{2}+y,2x+y)} \alpha_{(x+y,y)}\\
&=& \prod_{y\in X}\alpha_{(\frac{3x}{2}+y,2x+y)} \prod_{y\in X} \alpha_{(x+y,y)}
\end{eqnarray*}
By (\ref{alpha3}),
$\alpha_{(\frac{3x}{2}+y,2x+y)}=\alpha_{(\frac{3x}{2}+y,x+y)}=
\alpha_{(\frac{x}{2} +(x+y),x+y})$, for
all $x,y\in X$. Hence
\begin{eqnarray*}
 \prod_{y\in X}\alpha_{(\frac{3x}{2}+y,2x+y)}&=&  \prod_{y\in X}\alpha_{(\frac{x}{2}+x+y,x+y)}\; =\; \prod_{z\in X}\alpha_{(\frac{x}{2}+z,z)} .
\end{eqnarray*}
We thus get that
\begin{eqnarray*}\prod_{y\in X}\alpha_{(\frac{x}{2}+y,y)}&=&
 \prod_{y\in X}\alpha_{(\frac{3x}{2}+y,2x+y)} \prod_{y\in X} \alpha_{(x+y,y)}
 \; =\; \prod_{z\in X}\alpha_{(\frac{x}{2}+z,z)} \prod_{y\in X} \alpha_{(x+y,y)}
\end{eqnarray*}
and  thus
\begin{eqnarray*}
 \prod_{y\in X} \alpha_{(x+y,y)}&=&\id .
\end{eqnarray*}
This proves the Claim.

Let $k$ be a positive integer, let $m_1,m_2\in \{ 1,2\}$ and let
$x,y\in X$. We shall prove by induction on $k$ that
\begin{eqnarray}\label{2k}\lefteqn{s^{2k}((x+y,m_1),(y,m_2))}\\
&=&(((2k+1)x+y,\alpha_{(2kx+y,(2k-1)x+y)}\alpha_{((2k-2)x+y,(2k-3)x+y)}\cdots
\alpha_{(2x+y,x+y)}(m_1)),\notag\\
&&(2kx+y,\alpha_{((2k-1)x+y,(2k-2)x+y)}\alpha_{((2k-3)x+y,(2k-4)x+y)}\cdots
\alpha_{(x+y,y)}(m_2))).\notag\end{eqnarray}

For $k=1$, we have
\begin{eqnarray*}
s^2((x+y,m_1),(y,m_2))
&=&s((\sigma_{x+y}(y),\alpha_{(x+y,y)}(m_2)),(x+y,m_1))\\
&=&s((2x+y,\alpha_{(x+y,y)}(m_2)),(x+y,m_1))\\
&=&((\sigma_{2x+y}(x+y),\alpha_{(2x+y,x+y)}(m_1)),(2x+y,\alpha_{(x+y,y)}(m_2)))\\
&=&((3x+y,\alpha_{(2x+y,x+y)}(m_1)),(2x+y,\alpha_{(x+y,y)}(m_2))).
\end{eqnarray*}
Thus (\ref{2k}) is true for $k=1$. Suppose that $k>1$ and (\ref{2k})
holds for $k-1$. By the induction hypothesis and the definition of
$s$, we have
\begin{eqnarray*}
\lefteqn{s^{2k}((x+y,m_1),(y,m_2))}\\
&=&s^2(((2(k-1)+1)x+y,\\
&&\qquad
\alpha_{(2(k-1)x+y,(2(k-1)-1)x+y)}\alpha_{((2(k-1)-2)x+y,(2(k-1)-3)x+y)}\cdots
\alpha_{(2x+y,x+y)}(m_1)),\\
&&(2(k-1)x+y,
\alpha_{((2(k-1)-1)x+y,(2(k-1)-2)x+y)}\alpha_{((2(k-1)-3)x+y,(2(k-1)-4)x+y)}\cdots
\alpha_{(x+y,y)}(m_2)))\\
&=&s((\sigma_{(2(k-1)+1)x+y}(2(k-1)x+y),\\
&&\qquad\alpha_{((2(k-1)+1)x+y,2(k-1)x+y)}\alpha_{((2(k-1)-1)x+y,(2(k-1)-2)x+y)}\cdots
\alpha_{(x+y,y)}(m_2)),\\
&&((2(k-1)+1)x+y,\\
&&\qquad
\alpha_{(2(k-1)x+y,(2(k-1)-1)x+y)}\alpha_{((2(k-1)-2)x+y,(2(k-1)-3)x+y)}\cdots
\alpha_{(2x+y,x+y)}(m_1)))\\
&=&s(2kx+y,\alpha_{((2k-1)x+y,(2k-2)x+y)}\alpha_{((2k-3)x+y,(2k-4)x+y)}\cdots
\alpha_{(x+y,y)}(m_2)),\\
&&((2k-1)x+y,\alpha_{((2k-2)x+y,(2k-3)x+y)}\alpha_{((2k-4)x+y,(2(k-5)x+y)}\cdots
\alpha_{(2x+y,x+y)}(m_1)))\\
&=&((\sigma_{2kx+y}((2k-1)x+y),
\alpha_{(2kx+y,(2k-1)x+y)})\alpha_{((2k-2)x+y,(2k-3)x+y)}\cdots
\alpha_{(2x+y,x+y)}(m_1),\\
&&(2kx+y,
 \alpha_{((2k-1)x+y,(2k-2)x+y)}\alpha_{((2k-3)x+y,(2k-4)x+y)}\cdots
\alpha_{(x+y,y)}(m_2)))\\
&=&((2k+1)x+y,\alpha_{(2kx+y,(2k-1)x+y)}\alpha_{((2k-2)x+y,(2k-3)x+y)}\cdots
\alpha_{(2x+y,x+y)}(m_1),\\
&&(2kx+y,\alpha_{((2k-1)x+y,(2k-2)x+y)}\alpha_{((2k-3)x+y,(2k-4)x+y)}\cdots
\alpha_{(x+y,y)}(m_2))).
\end{eqnarray*}
Therefore (\ref{2k}) follows by induction.

Now we shall prove that
$s^{2p}((x+y,m_1),(y,m_2))=((x+y,m_1),(y,m_2))$ and therefore the
$s$-orbit $O_{((x+y,m_1),(y,m_2))}$ of $((x+y,m_1),(y,m_2))$ in
$Y^2$ has cardinality less than or equal to $2p$. By (\ref{2k})
\begin{eqnarray*}\lefteqn{s^{2p}((x+y,m_1),(y,m_2))}\\
&=&(((2p+1)x+y,\alpha_{(2px+y,(2p-1)x+y)}\alpha_{((2p-2)x+y,(2p-3)x+y)}\cdots
\alpha_{(2x+y,x+y)}(m_1)),\\
&&(2px+y,\alpha_{((2p-1)x+y,(2p-2)x+y)}\alpha_{((2p-3)x+y,(2p-4)x+y)}\cdots
\alpha_{(x+y,y)}(m_2))).
\end{eqnarray*}
Note that, by the Claim,
$$\prod_{k=1}^p\alpha_{(2kx+y,(2k-1)x+y)}=\prod_{k\in X}\alpha_{(2kx+y,(2k-1)x+y)}=\id$$
and
$$\prod_{k=1}^p\alpha_{((2k-1)x+y,(2k-2)x+y)}=\prod_{k\in X}\alpha_{((2k-1)x+y,(2k-2)x+y)}=\id.$$
Hence $s^{2p}((x+y,m_1),(y,m_2))=((x+y,m_1),(y,m_2))$. Therefore the
orbit $O_{((x+y,m_1),(y,m_2))}$  has cardinality less than or equal
to $2p$, as claimed. Now in $Y^2$ we have $2p$ singleton orbits of
the form $\{((x,s),(x,s))\}$  and $p$ orbits of the form
$$\{((x,1),(x,2)),((x,2),(x,1))\}.$$
Since $4p^2=4p+4p(p-1)$ and the cardinality of
any orbit is less than or equal to $2p$, we have at least
$2p+p+2(p-1)$ orbits. Since $2p+p+2(p-1)>4p-1=2|Y|-1$, the result
follows.
\end{proof}

We  now are ready to prove the main result of this section.

\begin{theorem} \label{minimal}
Let $(X,r)$ be a finite non-degenerate square-free $\SD$ braided
set. Suppose that $|X|>1$ and that the number of $r$-orbits in $X^2$
is $2|X|-1$. Then, up to isomorphism, one of the following holds
\begin{enumerate}
\item $|X|$ is an odd prime and $(X,r)$ is the braided set associated to the dihedral
quandle,
\item $|X|=2$ and $(X,r)$ is the trivial braided set,
\item $X=\{ 1,2,3\}$ and $\sigma_1=\sigma_{2}=\id$, $\sigma_3=(1,2)$.
\end{enumerate}
\end{theorem}
\begin{proof}
Suppose that $|X|=n$. In view of the comment after
Theorem~\ref{SquareFreeMinDim}, we may assume that $X=X_{\sigma_1}
\cup \cdots \cup X_{\sigma _r} \cup X_{\sigma_{r+1}} \cup \cdots
\cup X_{\sigma_k}$,  a disjoint union, where $|X_{\sigma_i}|=2$ for
$i=1,\ldots, r$ and $|X_{\sigma_i}|=1$ for $i=r+1,\ldots, k$. It
follows from the proof of Theorem~\ref{SquareFreeMinDim} that each
of the $2n-1$ $r$-orbits in $X^2$  is of one of the following types:
\begin{itemize}
\item  $\{(i,i)\}$, $i\in X$,
\item $\{ (i,i'), (i',i)\}$ where
$X_{\sigma_j}=\{ i,i'\}$, for $j=1,\ldots, r$, and $i\neq i'$,
\item $O_{(j,1)}$,  where $j\in X_{\sigma_{j}}$ for
$j=2,\ldots ,k$.
\end{itemize}
Recall that $\mbox{Ret}(X,r)$  also is a non-degenerate square-free
$\SD$ braided set $(Y,s)$. Note that $|Y|=k$. Because of the
previous, $Y^2$ is covered also  by $2k-1$ $s$-orbits and, by
Theorem~\ref{SquareFreeMinDim}, these are distinct orbits, so the
retract $\Ret (X,r)$ inherits the minimality assumption on $(X,r)$,
about the number of $r$-orbits in $X^2$.

Then, by Theorem~\ref{minimality}, there exists a positive integer
$m$ such that the $m$-th retract $\Ret^m(X,r)$ is either a trivial
braided set or it  is the braided set associated to a dihedral
quandle of odd prime order. We choose the minimal such number $m$.
If $m=0$ then the assertion follows from Theorem~\ref{minimality}.

So, assume that $m\geq 1$. Suppose first  that
$([X]_m,[r]_m)=\Ret^m(X,r)$ is the braided set associated to the
dihedral quandle of an odd prime order. Clearly $\mathcal{G}(X,r)$
acts transitively on $[X]_m$. Since $X_{\sigma_1} \cup \cdots \cup
X_{\sigma _r}$ is an orbit under the action of the group
$\mathcal{G}(X,r)$ it follows that $r=k$. However, this is in
contradiction with Lemma~\ref{extension}.

Thus, for the remainder of the proof we may assume that $m\geq 1$
and $([X]_m,[r]_m)=\Ret^m(X,r)$ is trivial. Hence $M([X]_m,[r]_m)$
is a free abelian monoid of rank $|[X]_m|$ and thus
$\mbox{dim}(A(K,[X]_m,[r]_m)_2)=\frac{|[X]_m| (|[X]_m|+1)}{2}$.
Because of Theorem~\ref{SquareFreeMinDim} (and the minimality of
$m$) we thus obtain that  $|[X]_m|=2$.

Thus $([X]_m,[r]_m)$ is trivial of cardinality $2$. Put
$([X]_{m-1},[r]_{m-1})=\Ret^{m-1}(X,r)$, where $[X]_0=X$. Since
$[X]_{m-1}$ also has a decomposition as outlined in the beginning of
the proof, one easily verifies that $([X]_{m-1},[r]_{m-1})$ either
is isomorphic to the braided set described in part (3) of the
statement of the result, or $[X]_{m-1}=\{ 1,1',2,2'\}$  and
$\sigma_1\neq\sigma_2$, $\sigma_1=\sigma_{1'}$ either is $\id$ or
$(2,2')$, and $\sigma_2=\sigma_{2'}$ either is $\id$ or $(1,1')$. In
the three cases, it is easy to see that we get  more  than $7$
$[r]_{m-1}$-orbits in $([X]_{m-1})^2$, and thus the minimality
condition is not satisfied in this case, a contradiction.

Hence, we are left to deal with the case that $[X]_{m-1}=\{ 1,2,3\}$
and with corresponding permutations $\sigma_1' =\sigma_2' =\id,
\sigma_3' =(1,2)$,  that is a braided set isomorphic to the braided
set described in part (3) of the statement of the result. If $m>1$
then there are the following possibilities for the set $[X]_{m-2}$,
where $([X]_{m-2},[r]_{m-2})=\Ret^{m-2}(X,r)$:
\begin{enumerate}
\item[(i)] $\{1,1',2,3\}$, or symmetrically $\{1,2,2',3\}$,
\item[(ii)] $\{1,1',2,2',3 \}$,
\item[(iii)] $\{1,2,3,3' \}$,
\item[(iv)] $\{1,2,2',3,3' \}$, or symmetrically $\{1,1',2,3,3' \}$,
\item[(v)] $\{1,1',2,2',3,3' \}$.
\end{enumerate}
Here, in each case we mean that $\sigma_{i}=\sigma_{i'}$ for every
$i$, and $\sigma_1,\sigma_2,\sigma_3$ are distinct.

Since $\sigma_1' =\id$, it follows that $\sigma_1 $ must be one of
the following permutations: $\id, (2,2'), (3,3')$, or $(2,2')(3,3')$.
Similarly, since $\sigma_2 =\id$, we get that $\sigma_2$ is one of
the following permutations: $\id, (1,1'),(3,3')$, or $(1,1')(3,3')$.
Finally, it is easy to see that $\sigma_3 $ has one of the following forms: $ (1,2)$ or
$(1,2)(1',2')$ or $(1,2,1',2')$ (or maybe $1,1'$ and/or $2,2'$
have to be switched).

So,  by Lemma~\ref{conditions}, we always have $\sigma_3\sigma_1 =
\sigma_2 \sigma_3$, because $\sigma_{3}(1)=2$ or $2'$. Thus, only
cases (ii) or (v) are possible ($\sigma_1, \sigma_2$ are conjugate
by $\sigma_3$ but not equal). But,  again by Lemma~\ref{conditions},
$\sigma_1 \sigma_3= \sigma_3 \sigma_1$ because $\sigma_1(3)=3$ or
$3'$. Then $\sigma_1 =\sigma_2$, which is a contradiction. So there
are no braided sets of this type. Therefore, $m=1$ and $(X,r)$ is of
type as described in part (3) of the statement of the result. This
completes the proof of the theorem.
\end{proof}

Note that the dihedral quandles of cardinality an odd prime are
simple quandles (see \cite{J2}).  Thus this also answers Open
Questions~4.3.1 (1) and (2) in \cite{GI2018}. On the other hand, the
non-singleton orbits in $X^3$ under the action of the group $\langle
r_1,r_2\rangle$ in the above type (3) have not the same cardinality.
In fact, $|O_{(1,1,3)}|=12$, $|O_{(1,3,3)}|=6$ and
$|O_{(1,1,2)}|=|O_{1,2,2}|=3$. Therefore this answers in the
negative Open Question 4.1.2 in \cite{GI2018}.

As before, using Proposition~1.4 in \cite{JKV}, we  derive the
main result of this paper, as the following immediate consequence
of Theorem~\ref{minimal}. This  solves another problem of
Gateva-Ivanova (Question~3.18 (2) in \cite{GI2018}).

\begin{corollary} \label{cor-minimal} Let $(X,r)$ be a finite
square-free non-degenerate braided set (not necessarily $\SD$) such
that $|X|>1$ and the number of $r$-orbits in $X^2$ is $2|X|-1$. Then
$|X|$ is a prime number.  Furthermore, its derived solution $(X,r')$
is one of the braided sets described in the statement of
Theorem~\ref{minimal}  and $\dim(A(K,X,r)_m)=\dim(A(K,X,r')_m)$
for every nonnegative integer $m$ and any field $K$.
\end{corollary}

\section{Solutions extending cycles}

In \cite{GI2018} Gateva-Ivanova provided several examples of
2-cancellative $\SD$ square-free solutions $(X, r)$  that contain
a given $r$-orbit of cardinality $|X|$. In this section we prove
that such solutions always  exist. Furthermore, if $|X|$ is odd,
we prove that there is precisely one solution. Hence, we give a
complete solution of Problem 4.3.2. in \cite{GI2018}.

The following well-known lemma will be needed.

\begin{lemma}\label{cycle}
Let $X$ be a set of cardinality $n>1$. Let $(x_1,\dots, x_n),
(y_1,\dots, y_n)\in \Sym_X$ be cyclic permutations of length $n$.
Then the permutation $\sigma\in \Sym_X$ defined by
$\sigma(x_i)=y_i$, for all $1\leq i\leq n$, is the unique
permutation such that $\sigma(x_1)=y_1$ and $\sigma (x_1,\dots,
x_n)\sigma^{-1}=(y_1,\dots, y_n)$.
\end{lemma}

Here is our final result.

\begin{proposition}\label{sd}
Let $n$ be an integer greater than $1$. Let $X=\{ b_1,\dots ,b_n\}$
be a set of cardinality $n$. Let $\mathcal{O}=\{ (a_i,b_i)\mid 1\leq
i\leq n\}$ be a subset of $X^2$. Let $r_0\in\Sym_{\mathcal{O}}$ be a
cyclic permutation of order $n$:
$$(a_1,b_1)\mapsto_{r_0}(a_2,b_2)\mapsto_{r_0}\cdots\mapsto_{r_0}(a_n,b_n)\mapsto_{r_0}(a_1,b_1),$$
where $\{ a_1,\dots ,a_n\}=X$ and $a_i\neq b_i$, for $1\leq i\leq
n$. Then there exists an extension $r\colon X\times X\longrightarrow
X\times X$ of $r_0$ such that $(X,r)$ is a $2$-cancellative
square-free $\SD$ braided set, with $r(x,y)=(\sigma_x(y),x)$ and
$\sigma_x^2=\id$, for all $x\in X$, if and only if $a_n=b_1$ and
$a_i=b_{i+1}$ for $1\leq i\leq n-1$. Furthermore, in this case, if
$n$ is odd then $(X,r)$  is unique.
\end{proposition}

\begin{proof} It is clear that if $(X,r)$ is a
$2$-cancellative square-free $\SD$ braided set, with
$r(x,y)=(\sigma_x(y),x)$ and $\sigma_x^2=\id$, for all $x\in X$, and
the restriction of $r$ to $\mathcal{O}$ is $r_0$, then
 $a_n=b_1$ and $a_i=b_{i+1}$ for $1\leq i\leq n-1$.

Suppose that  $a_n=b_1$ and $a_i=b_{i+1}$ for $1\leq i\leq n-1$. We
may assume  that $X=\{1,2,\ldots ,n\}$ and $b_i=i$, for all $i$.
Hence
$$(2,1)\mapsto_{r_0}(3,2)\mapsto_{r_0}(4,3)\mapsto_{r_0}\cdots\mapsto_{r_0}(n,n-1)\mapsto_{r_0}(1,n)\mapsto_{r_0}(2,1).$$

We shall prove that there exists a $2$-cancellative square-free
$\SD$ braided set $(X,r)$, with $r(x,y)=(\sigma_x(y),x)$ and
$\sigma_x^2=\id$, for all $x\in X$, such that the restriction of $r$
to $\mathcal{O}$ is $r_0$.

Note that the $\sigma_i$ should be elements in $\Sym_X$ such that
$\sigma_i(i-1)=i+1$, for $2\leq i\leq n-1$, $\sigma_n(n-1)=1$ and
$\sigma_1(n)=2$. Furthermore, by Lemma~\ref{conditions} and since
$\sigma _j^2=\id$,  we have
$$\sigma_i=\sigma_{i-1}\sigma_{i-2}\sigma_{i-1},$$
for all $3\leq i\leq n$, $\sigma_1=\sigma_n\sigma_{n-1}\sigma_n$ and
$\sigma_2=\sigma_1\sigma_n\sigma_1$. We shall prove by induction on
$i$ that
\begin{eqnarray}\label{sigma}
&&\sigma_i=(\sigma_2\sigma_1)^{i-2}\sigma_2, \end{eqnarray} for all
$3\leq i\leq n$. For $i=3$, the result is true. For $i=4\leq n$,
$$\sigma_4=\sigma_3\sigma_2\sigma_3
=\sigma_2\sigma_1\sigma_2\sigma_2\sigma_2\sigma_1\sigma_2
=(\sigma_2\sigma_1)^2\sigma_2.
$$
Suppose that $4<i\leq n$ and that (\ref{sigma}) is true for $i-1$
and $i-2$. Now, by the induction hypothesis, we have
\begin{eqnarray*}
\sigma_i&=&\sigma_{i-1}\sigma_{i-2}\sigma_{i-1}\\
&=&(\sigma_2\sigma_1)^{i-3}\sigma_2(\sigma_2\sigma_1)^{i-4}\sigma_2(\sigma_2\sigma_1)^{i-3}\sigma_2\\
&=&(\sigma_2\sigma_1)^{i-3}\sigma_2\sigma_2(\sigma_1\sigma_2)^{i-4}(\sigma_2\sigma_1)^{i-3}\sigma_2\\
&=&(\sigma_2\sigma_1)(\sigma_2\sigma_1)^{i-3}\sigma_2\\
&=&(\sigma_2\sigma_1)^{i-2}\sigma_2.
\end{eqnarray*}
Therefore (\ref{sigma}) follows by induction. Since
$\sigma_1=\sigma_n\sigma_{n-1}\sigma_n$,  a similar argument shows
that
$$\sigma_1=(\sigma_2\sigma_1)^{n-1}\sigma_2.$$
Hence $(\sigma_2\sigma_1)^n=\id$.

Since $(X,r)$ should be square-free, by Corollary~\ref{rn}, we
have that
\begin{equation}\label{eq3}(2k,2k-1)=(a_{2k-1},b_{2k-1})=r^{2(k-1)}(a_1,b_1)=((\sigma_2\sigma_1)^{k-1}(2),(\sigma_2\sigma_1)^{k-1}(1))\end{equation}
and
\begin{equation}\label{eq4}(2k+1,2k)=(a_{2k},b_{2k})=r^{2k-1}(a_1,b_1)=((\sigma_2\sigma_1)^{k}(1),(\sigma_2\sigma_1)^{k-1}(2)),\end{equation}
for all $2\leq k\leq (n-1)/2$. Let
$\sigma=\sigma_2\sigma_1\in\Sym_n$. Note that $\sigma_2\sigma
\sigma_{2}^{-1}=\sigma_2\sigma\sigma_2=\sigma_1\sigma_2=\sigma^{-1}$.

Suppose that $n$ is odd. In this case, $n=2t+1$, for some $t$, and
the above conditions (\ref{eq3}) and (\ref{eq4}) imply that
$$\sigma=(2,4,6,\dots,2t,1,3,5,\dots, 2t+1).$$
Since $\sigma_2(2)=2$ and $\sigma_2\sigma\sigma_2^{-1}=\sigma^{-1}$,
by Lemma~\ref{cycle}, $\sigma_2=(4,2t+1)(6,2t-1)\cdots (2t,5)(1,3)$.
Thus $(X,r)$ is the solution corresponding to the dihedral quandle.

Suppose that $n$ is even. In this case, $n=2t$, for some $t$, and
(\ref{eq3}) and (\ref{eq4}) imply
$$\sigma=(2,4,6,\dots ,2t)(1,3,5,\dots,2t-1).$$
Note that $\sigma^{-1}=(2,2t,2t-2,\dots ,6,4)(2t-1,2t-3,\dots
,3,1)$. By Lemma~\ref{cycle}, since $\sigma_2(2)=2$, the restriction
of $\sigma_2$ to $\{ 2,4,\dots ,2t\}$ should be equal to
$$(4,2t)(6,2(t-1))\cdots
(2\left\lfloor\frac{t+1}{2}\right\rfloor,2\left\lceil\frac{t+3}{2}\right\rceil).$$
Also by Lemma~\ref{cycle}, there are $t$ possibilities for the
restriction of $\sigma_2$ to $\{ 1,3,5,\dots, 2t-1\}$ depending on
the value of $\sigma_2(1)\in \{ 1,3,5,\dots, 2t-1\}$. Thus,
$\sigma_2$ is one of the following permutations.
$$\sigma_{2,j}=(4,2t)(6,2(t-1))\cdots
(2\left\lfloor\frac{t+1}{2}\right\rfloor,2\left\lceil\frac{t+3}{2}\right\rceil)\tau_j,$$
where $\tau_j(x)=x$, for all $x\in\{2,4,6,\dots, 2t\}$,  and
$$\tau_k(1)=2k-1, \tau_k(3)=2k-3,\dots
,\tau_k(2k-1)=1,$$
$$\tau_k(2k+1)=2t-1,\tau_k(2k+3)=2t-3,\dots
,\tau_k(2t-1)=2k+1,$$  for $1\leq k<t$; and
$$\tau_t(1)=2t-1, \tau_t(3)=2t-3,\dots
,\tau_t(2t-1)=1.$$ We shall check for which of these $t$
possibilities $(X,r)$ is a braided set. By Lemma~\ref{conditions} we
should check  when
$$\sigma_x\sigma_y=\sigma_{\sigma_x(y)}\sigma_x,$$
for all $x,y\in X$,  because $\gamma_y =id$ for every $y\in X$.

By (\ref{sigma})
$$\sigma_x\sigma_y=\sigma^{x-2}\sigma_2\sigma^{y-2}\sigma_2=\sigma^{x-y},$$
for all $1\leq x,y\leq n$. If $y$ is even and $4\leq y\leq 2t$, then
$$\sigma_x(y)
=\sigma^{x-2}\sigma_2(y) =\sigma^{x-2}(2t-y+4) \equiv
2t-y+4+2(x-2)\equiv 2x-y\; (\Mod 2t),
$$
and we have that
$$\sigma_{\sigma_x(y)}\sigma_x=\sigma^{2x-y-2}\sigma_2\sigma^{x-2}\sigma_2=\sigma^{x-y}.$$
If $y=2$, then
$$\sigma_x(2)
=\sigma^{x-2}\sigma_2(2) =\sigma^{x-2}(2) \equiv  2+2(x-2)\equiv
2x-2\; (\Mod 2t),
$$
and we have that
$$\sigma_{\sigma_x(2)}\sigma_x=\sigma^{2x-2-2}\sigma_2\sigma^{x-2}\sigma_2=\sigma^{x-2}.$$
Hence
$$\sigma_x\sigma_y=\sigma_{\sigma_x(y)}\sigma_x,$$
for all $x,y\in X$ with $y$ even.

Suppose that $y$ is odd and $\sigma_2(1)=2j-1$. We have
$$\sigma_x(y)
=\sigma^{x-2}\sigma_2(y) \equiv 2(x-2)+\tau_j(y)\equiv 2(x-2)+2j-y\;
(\Mod 2t)
$$
and  we have that
$$\sigma_{\sigma_x(y)}\sigma_x=\sigma^{2x+2j-y-6}\sigma_2\sigma^{x-2}\sigma_2=\sigma^{x+2j-y-4}.$$
Thus
$$\sigma_x\sigma_y=\sigma_{\sigma_x(y)}\sigma_x,$$
if and only if $\sigma^{x-y}=\sigma^{x-y+2j-4}$.  Hence $(X,r)$ is a
braided set if and only if $\sigma_2=\sigma_{2,j}$ and $2j-4\equiv 0
\; (\Mod t)$, because $\sigma $ has order $t$. If $t$ is odd, then
this happens if and only if $j=2$. If $t$ is even, then this happens
if and only if either $j=2$ or $j=2+t/2$.

Hence we have proved that if $n\equiv 2\; (\Mod 4)$, then $(X,r)$
with $r(x,y)=(\sigma_x(y),x)$, $\sigma_2=\sigma_{2,2}$,
$\sigma=(2,4,6,\dots ,2t)(1,3,5,\dots,2t-1)$ and
$\sigma_x=\sigma^{x-2}\sigma_2$, for all $x\in X$, is the unique
square-free $\SD$ braided set such that the restriction of $r$ to
$\mathcal{O}$ is $r_0$.

We have also proved that if $n\equiv 0\; (\Mod 4)$, then $(X,r)$
with $r(x,y)=(\sigma_x(y),x)$, $\sigma_2=\sigma_{2,2}$,
$\sigma=(2,4,6,\dots ,2t)(1,3,5,\dots,2t-1)$ and
$\sigma_x=\sigma^{x-2}\sigma_2$, for all $x\in X$, is a square-free
$\SD$ braided set such that the restriction of $r$ to $\mathcal{O}$
is $r_0$. But in this case, we also obtain another square-free $\SD$
braided set $(X,s)$  such that the restriction of $s$ to
$\mathcal{O}$ is $r_0$, with $s(x,y)=(\sigma_x(y),x)$,
$\sigma_2=\sigma_{2,2+t/2}$ and $\sigma=(2,4,6,\dots
,2t)(1,3,5,\dots,2t-1)$ and $\sigma_x=\sigma^{x-2}\sigma_2$, for all
$x\in X$.

To finish the proof it remains to prove that the three types of
braided sets that we have obtained are $2$-cancellative.

Let $(X,r)$ be one of the three types of braided sets  that we
have obtained with $|X|=n$. We know that
$\sigma_x\sigma_y=\sigma^{x-y}$. Since $(X,r)$ is square-free
non-degenerate, by Corollary~\ref{rn}, for every positive integer
$k$,
\begin{eqnarray*}
r^{2k-1}(x,y)&=&((\sigma_x\sigma_y)^k(y),(\sigma_x\sigma_y)^{k-1}(x))=(\sigma^{k(x-y)}(y),\sigma^{(k-1)(x-y)}(x)),\\
r^{2k}(x,y)&=&((\sigma_x\sigma_y)^k(x),(\sigma_x\sigma_y)^{k}(y))=(\sigma^{k(x-y)}(x),\sigma^{k(x-y)}(y)).
\end{eqnarray*}
Assume for example that $xy=xy'$ in $M(X,r)$, for some $x,y,y'\in X$
(the other case is symmetric). Then we have to look at two possible
cases: either $(x,y')= r^{2k-1}(x,y)$ or  $(x,y') = r^{2k}(x,y)$,
for some $k$. In the former case this means that $\sigma^{k(x-y)}
(y)= x$. And we need to prove that $y'=\sigma^{(k-1)(x-y)} (x)= y$.
But $x= \sigma^{k(x-y)}(y)$ means (by the definition of $\sigma$)
that $x\equiv  y + 2k(x-y)\; (\Mod n)$. Then $y\equiv 2(k-1)(x-y)
+x\; (\Mod n)$, which implies that  $y= \sigma^{(k-1)(x-y)} (x)$.
And in the latter case we have $\sigma^{k(x-y)} (x) =x$ and we need
to prove that $y'= \sigma^{k(x-y)}(y) =y$. So, by the definition of
$\sigma$, $x\equiv x+2k(x-y)\; (\Mod n)$ and then $y'\equiv
y+2k(x-y)\equiv y\; (\Mod n)$. Hence $\sigma^{k(x-y)}(y) =y$, as
desired.

So $2$-cancellativity follows and the proof of the proposition is
completed.
\end{proof}

\vspace{30pt}
 \noindent \begin{tabular}{llllllll}
  F. Ced\'o && E. Jespers\\
 Departament de Matem\`atiques &&  Department of Mathematics \\
 Universitat Aut\`onoma de Barcelona &&  Vrije Universiteit Brussel \\
08193 Bellaterra (Barcelona), Spain    && Pleinlaan 2, 1050 Brussel, Belgium \\
 cedo@mat.uab.cat && Eric.Jespers@vub.be \\ \\
 J. Okni\'{n}ski && \\ Institute of
Mathematics &&
\\  Warsaw University &&\\
 Banacha 2, 02-097 Warsaw, Poland &&\\
 okninski@mimuw.edu.pl &&
\end{tabular}

 \end{document}